%% file: cf_survey.tex
\documentclass[11pt]{article}

\usepackage{a4wide}
\usepackage{algorithm,algorithmic}
\usepackage{amsfonts}
\usepackage{amsbsy}
\usepackage{amssymb,amsmath,amsthm}
\usepackage{amstext}
\usepackage{graphicx,xcolor}

\advance\textwidth 2cm
\advance\textheight 2cm
\advance\oddsidemargin -1cm
\advance\evensidemargin -1cm
\advance\topmargin -1cm

\def\inline#1:{\par\vskip 7pt\noindent{\bf #1:}\hskip 10pt}

\def\ZZ{{\mathbb Z}}

\newcommand{\bfm}{\textbf}
\newcommand{\mcal}{\mathcal}
\newcommand{\mbb}{\mathbb}

\newcommand{\E}{\mcal{E}}
\newcommand{\R}{\mcal{R}}
\renewcommand{\Re}{\mbb{R}}
\newcommand{\B}{\mcal{B}}
\def\reals{{\mathbb R}}

\newcommand{\D}{\mcal{D}}
\newcommand{\kopt}{\chi_{kCF}}

\newcommand{\chum}{{{ch_{\text{\textup{um}}}}}}
\newcommand{\chcf}{{{ch_{\text{\textup{cf}}}}}}
\newcommand{\umch}{{{ch_{\text{\textup{um}}}}}}

\newcommand{\A}{\mcal{A}}
\newcommand{\cf}{{{\chi_{\text{\textup{cf}}}}}}
\newcommand{\CF}{{{\chi_{\text{\textup{cf}}}}}}

\newcommand{\um}{{{\chi_{\text{\textup{um}}}}}}

\newcommand{\Alg}{\mathbf{A}}

\newcommand{\lemlab}[1]{\label{lemma:#1}}
\newcommand{\lemref}[1]{Lemma~\ref{lemma:#1}}
\newcommand{\seclab}[1]{\label{section:#1}}
\newcommand{\secref}[1]{Section~\ref{section:#1}}
\newcommand{\theolab}[1]{\label{theorem:#1}}
\newcommand{\theoref}[1]{Theorem~\ref{theorem:#1}}
\newcommand{\cfch}{{{ch_{cf}}}}
\newcommand{\ch}{{{ch}}}

\newtheorem{theorem}{Theorem}[section]
\newtheorem{lemma}[theorem]{Lemma}
\newtheorem{corollary}[theorem]{Corollary}
\newtheorem{definition}[theorem]{Definition}
\newtheorem{proposition}[theorem]{Proposition}

\newtheorem{problem}{Problem}

\theoremstyle{remark}
\newtheorem{remark}[theorem]{Remark}
\newcommand{\cardin}[1]{\left| {#1} \right|}
\newcommand{\UComp}{{\mathcal{U}}}
\newcommand{\Arr}{\mathop{\mathrm{\A}}}
\newcommand{\pth}[1]{\!\left({#1}\right)}

\newcommand{\eps}{\varepsilon}
\newcommand{\Rcal}{\mathcal{R}}
\newcommand{\chiopt}{\chi_{\rm opt}}
\newcommand{\calC}{\mathcal{C}}
\newcommand{\chC}{{{ch_{\calC}}}}
\newcommand{\chicf}{{{\chi_{\text{\textup{cf}}}}}}

\newcommand{\chiC}{{{\chi_{\calC}}}}

\newcommand*{\oneton}{\{1,\dots,n\}}%
\newcommand*{\Positives}{\ensuremath{\Nats^+}}%
\newcommand*{\Nats}{\mathbb{N}}%
\newcommand{\Ex}{\mathbf{E}}
\renewcommand{\L}{\mathcal{L}}

\begin{document}
\title{Conflict-Free Coloring and its Applications}
\author{
Shakhar Smorodinsky\thanks{Mathematics department, Ben-Gurion
University, Beer Sheva, Israel; {\tt
http://www.math.bgu.ac.il/$\sim$shakhar/ ;
shakhar@math.bgu.ac.il}.}}

 \maketitle

\begin{abstract}
Let $H=(V,E)$ be a hypergraph. A {\em conflict-free} coloring of
$H$ is an assignment of colors to $V$ such that, in each hyperedge
$e \in E$, there is at least one uniquely-colored vertex. This
notion is an extension of the classical graph coloring. Such
colorings arise in the context of frequency assignment to cellular
antennae, in battery consumption aspects of sensor networks, in
RFID protocols, and several other fields. Conflict-free coloring has been the focus of
many recent research papers. In this paper, we survey this notion
and its combinatorial and algorithmic aspects.
\end{abstract}

\section{Introduction}
\label{sec:intro}

\subsection{Notations and Definitions}
In order to introduce the main notion of this paper, we start
with several basic definitions:
Unless otherwise stated, the term $\log$ denotes the base $2$ logarithm.

A {\em hypergraph} is a pair
$(V,\E$) where $V$ is a set and $\E$ is a collection of subsets
of $V$. The elements of $V$ are called {\em vertices} and the
elements of $\E$ are called {\em hyperedges}. When all hyperedges
in $\E$ contain exactly two elements of $V$ then the pair
$(V,\E)$ is a {\em simple graph}. For a subset $V' \subset V$
refer to the hypergraph $H(V') = (V',\{S\cap V'| S \in \E\})$ as
the {\em sub-hypergraph} induced by $V'$. A $k$-coloring, for some
$k \in \mathbb N$, of (the vertices of) $H$ is a function
$\varphi:V \rightarrow \{1,\ldots,k\}$. Let $H=(V,\E)$ be a
hypergraph. A $k$-coloring $\varphi$ of $H$ is called {\em
proper} or {\em non-monochromatic} if every hyperedge $e \in \E$
with $\cardin{e} \geq 2$ is non-monochromatic. That is, there
exists at least two vertices $x,y \in e$ such that $\varphi(x)
\neq \varphi(y)$. Let $\chi(H)$ denote the least integer $k$ for
which $H$ admits a proper coloring with $k$ colors.

In this paper, we focus on the following colorings which are more
restrictive than proper coloring:

\begin{definition}[Conflict-Free and Unique-Maximum Colorings]
Let $H=(V,\E)$ be a hypergraph and let $C:V \rightarrow
\{1,\ldots,k\}$ be some coloring of $H$. We say that $C$ is a {\em
conflict-free} coloring ({\em CF-coloring} for short) if every
hyperedge $e \in \E$ contains at least one uniquely colored
vertex. More formally, for every hyperedge $e \in \E$ there is
a vertex $x \in e$ such that $\forall y \in e, y\neq x \Rightarrow
C(y) \neq C(x)$. We say that $C$ is a {\em unique-maximum} coloring ({\em UM-coloring} for short) if the maximum
color in every hyperedge is unique. That is, for every hyperedge $e \in \E$,
$\cardin{e \cap C^{-1}(\max_{v \in e}C(v)) }=1$.
\end{definition}
Let $\CF(H)$ (respectively, $\um(H)$) denote the least integer $k$ for which $H$ admits a
CF-coloring (respectively, a UM-coloring) with $k$ colors. Obviously, every UM-coloring of a hypergraph $H$ is also a CF-coloring of $H$, and every CF-coloring of $H$ is also a proper coloring of $H$. Hence, we have the followng inequalities:
$$
\chi(H) \leq \cf(H) \leq \um(H)
$$
Notice that for simple graphs, the three notions of coloring (non-monochromatic, CF and UM) coincide.
Also, for $3$-uniform hypergraphs (i.e., every hyperedge has cardinality $3$), the two first notions (non-monochromatic and CF) coincide. However, already for $3$-uniform hypergraphs there can be an arbitrarily large gap
 between $\cf(H)$ and $\um(H)$. Consider, for example, two sets $A$ and $B$ each of cardinality $n>1$.
 Let $H=(A \cup B, \E)$
 where $\E$ consists of all triples of elements $e$ such that $e \cap A \neq \emptyset$ and $e \cap B \neq \emptyset$.
 In other words $\E$ consists of all triples containing two elements from one of the sets $A$ or $B$ and one element from the other set. It is easily seen that $\cf(H)=2$ by simply coloring all elements of $A$ with $1$
  and all elements of $B$ with $2$. It is also not hard to verify that $\um(H) \geq n$ (in fact $\um(H)=n+1$).
  Indeed, let $C$ be a UM-coloring of $H$. If all elements of $A$ are colored with distinct colors we are done. Otherwise, there exist
   two elements $u,v$ in $A$ with the same color, say $i$. We claim that all elements of $B$ are colored with colors greater than $i$. Assume to the contrary that there is an element $w \in B$ with color $C(w)= j \leq i$. However, in that case the hyperedge $\{u,v,w\}$ does not have the unique-maximum property. Hence all colors
   of $B$ are distinct for otherwise if there are two vertices $w_1,w_2$ with the same color, again the hyperedge $\{w_1,w_2,u\}$ does not have the unique-maximum property.

Let us describe a simple yet an important example of a
hypergraph $H$ and analyze its chromatic number $\chi(H)$ and its
CF-chromatic number $\CF(H)$. The vertices of the hypergraph
consist of the first $n$ integers $[n]=\{1,\ldots,n\}$. The
hyperedge-set is the set of all (non-empty) subsets of $[n]$
consisting of consecutive elements of $[n]$, e.g., $\{2,3,4\}$,
$\{2\}$, the set $[n]$, etc. We refer to such hypergraphs as
{\em hypergraphs induced by points on the line with respect to
intervals} or as the {\em discrete intervals hypergraph}. Trivially, we
have $\chi(H) = 2$. We will prove the following proposition:

\begin{proposition}
 $\CF(H)= \um(H) =\lfloor \log n
\rfloor + 1$.
\end{proposition}

\begin{proof}
First we prove that $\um(H) \leq \lfloor \log n
\rfloor + 1$.
Assume without loss of generality that
$n$ is of the form $n=2^k-1$ for some integer $k$. If $n < 2^k-1$
then we can add the vertices $n+1,n+2,\ldots,2^k-1$ and this can
only increase the CF-chromatic number. In this case we will see
that $\um(H) \leq k$ and that for $n \geq 2^k$ $\CF(H) \geq k+1$. The
proof is by induction on $k$. For $k=1$ the claim holds trivially.
Assume that the claim holds for some integer $k$ and let
$n=2^{k+1}-1$. Consider the median vertex $2^k$ and color it with
a unique (maximum color), say $k+1$, not to be used again. By the induction
hypothesis, the set of elements to the right of $2^k$, namely the
set $\{2^k+1,2^k+2,\ldots,2^{k+1}-1\}$ can be colored with $k$
colors, say `1',`2'...,`$k$', so that any of its subsets of
consecutive elements has unique maximum color. The same holds
for the set of elements to the left of $2^k$. We will use the
same set of $k$ colors for the right set and the left set (and
color the median with the unique color `k+1'). It is easily verified
that this coloring is indeed a UM-coloring for $H$. Thus we use a
total of $k+1$ colors and this completes the induction step.

Next, we need to show
that for $n \geq 2^k$ we have $\CF(H) \geq k+1$. Again, the proof is by induction on $k$. The base case $k=0$ is trivial.
For the induction step, let $k > 0$ and put $n=2^k$. Let $C$ be some CF-coloring of the underlying discrete intervals hypergraph. Consider the hyperedge $[n]$. There must be a uniquely colored vertex in $[n]$.
Let $x$ be this vertex. Either to the right of $x$ or to its left we have at least $2^{k-1}$ vertices. That is, there is a hyperedge $S \subset [n]$ that does not contain $x$ such that $\cardin{S} \geq 2^{k-1}$, so, by the induction hypothesis, any CF-coloring for $S$ uses at least $k$ colors. Thus, together with the color of $x$, $C$ uses at least $k+1$ colors in total. This completes the induction step.
\end{proof}

 The notion of CF-coloring was first
introduced and studied in \cite{SmPHD} and \cite{ELRS}. This
notion attracted many researchers and has been the focus of many
research papers both in the computer science and mathematics
communities. Recently, it has been studied also in the infinite settings of the so-called {\em almost disjoint set systems} by Hajnal et al. \cite{hajnal}.
In this survey, we mostly consider hypergraphs that naturally arise in geometry. These come
in two types:

\begin{itemize}
\item
{\bf \noindent Hypergraphs induced by regions:}  Let $\R$ be a
finite collection of regions (i.e., subsets) in $\mbb{R}^d$, $d
\geq 1$. For a point $p \in \mbb{R}^d$, define $r(p) = \{R \in
\R: p \in R\}$. The hypergraph $(\R, \{r(p)\}_{p\in \mbb{R}^d})$,
denoted $H(\R)$, is called the {\em hypergraph \emph{induced} by $\R$}. Since $\R$ is finite, so is the power set $2^{\R}$. This implies that the hypergraph $H(\R)$ is finite as well.
\item
{\bf \noindent Hypergraphs induced by points with respect to
regions:} Let $P \subset \Re^d$ and let $\R$ be a family of
regions in $\Re^d$. We refer to the hypergraph $H_{\R}(P) = (P,\{ P\cap S| S
\in \R\})$ as the {\em hypergraph induced by $P$ with respect to $\R$}.
When $\R$ is clear from the context we sometimes refer to it as
{\em the hypergraph induced by $P$}. In the literature, hypergraphs
that are induced by points with respect to geometric regions of
some specific kind are sometimes referred to as {\em range
spaces}.
\end{itemize}

\begin{definition}[Delaunay-Graph]
For a hypergraph $H=(V,\E)$, denote by $G(H)$ the {\em
Delaunay-graph} of $H$ which is the graph $(V,\{S\in \E \mid
\cardin{S} = 2\})$.
\end{definition}

In most of the
coloring solutions presented in this paper we will see that, in
fact, we get the stronger UM-coloring. It is also interesting to study hypergraphs for which $\CF(H) < \um(H)$.
This line of research has been pursued
in \cite{CKP10,ChT2010ciac}

\subsection{Motivation}
We start with several motivations for studying CF-colorings and UM-colorings.
\subsubsection{Wireless Networks}
Wireless communication is used in many different situations such
as mobile telephony, radio and TV broadcasting, satellite
communication, etc. In each of these
situations a frequency assignment problem arises with
application-specific characteristics. Researchers have developed
different modeling approaches for each of the features of the
problem, such as the handling of interference among radio
signals, the availability of frequencies, and the optimization
criterion.

The work of Even et al. \cite{ELRS} and of Smorodinsky
\cite{SmPHD} proposed to model frequency assignment to cellular
antennas as  CF-coloring. In this new model, one can use a very
``small" number of distinct frequencies in total, to assign to a
large number of antennas in a wireless network. Cellular networks
are heterogeneous networks with two different types of nodes:
{\em base-stations} (that act as servers) and {\em clients}. The
base-stations are interconnected by an external fixed backbone
network. Clients are connected only to base stations; connections
between clients and base-stations are implemented by radio links.
Fixed frequencies are assigned to base-stations to enable links
to clients. Clients, on the other hand, continuously scan
frequencies in search of a base-station with good reception. This
scanning takes place automatically and enables smooth transitions
between base-stations when a client is mobile. Consider a client that is
within the reception range of two base stations. If these two
base stations are assigned the same frequency, then mutual
interference occurs, and the links between the client and each of
these conflicting base stations are rendered too noisy to be
used. A base station may serve a client provided that the
reception is strong enough and interference from other base
stations is weak enough. The fundamental problem of frequency
assignment in cellular network is to assign frequencies to base
stations so that every client is served by some base station. The
goal is to minimize the number of assigned frequencies since the
available spectrum is limited and costly.

The problem of frequency assignment was traditionally treated as a
graph coloring problem, where the vertices of the graph are the
given set of antennas and the edges are those pairs of antennas
that overlap in their reception range. Thus, if we color the
vertices of the graph such that no two vertices that are
connected by an edge have the same color, we guarantee that there
will be no conflicting base stations. However, this model is too
restrictive. In this model, if a client lies within the reception
range of say, $k$ antennas, then every pair of these antennas are
conflicting and therefore they must be assigned $k$ distinct
colors (i.e., frequencies). But note that if one of these
antennas is assigned a color (say $1$) that no other antenna is
assigned (even if all other antennas are assigned the same color,
say $2$) then we use a total of two colors and this client can
still be served. See Figure~\ref{example} for an illustration with
three antennas.

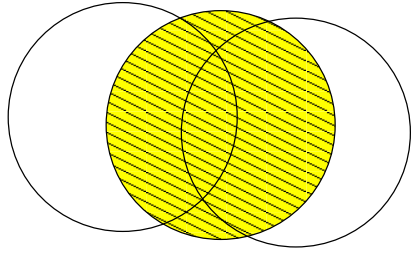
\begin{figure}[htb]
\begin{center}
\input{example.pstex_t}
\caption{An example of three antennas presented as discs in the
plane. In the classical model three distinct colors are needed
where as in the new model two colors are enough as depicted here.}
\label{example}
\end{center}
\end{figure}

A natural question thus arises: Suppose we are given a set of $n$
antennas. The location of each antenna (base station) and its
radius of transmission is fixed and is known (and is modeled as a
disc in the plane). We seek the least number of colors that
always suffice such that each of the discs is assigned one of the
colors and such that every covered point $p$  is also covered by
some disc $D$ whose assigned color is distinct from all the
colors of the other discs that cover $p$. This is a special case
of CF-coloring where the underlying hypergraph is induced by a
finite family of discs in the plane.

\subsubsection{RFID networks}

Radio frequency identification (RFID) is a technology where a
reader device can "sense" the presence of a close by object by
reading a tag device attached to the object. To improve coverage,
multiple RFID readers can be deployed in the given region. RFID
systems consist of readers and tags. A tag has an ID stored in
its memory. The reader is able to read the IDs of the tags in the
vicinity by using wireless protocol. In a typical RFID
application, tags are attached to objects of interest, and the
reader detects the presence of an object by using an available
mapping of IDs to objects. We focus on passive tags i.e., tags
that do not carry a battery.  The power needed for passive tags
to transmit their IDs to the reader is "supplied" by the reader
itself. Assume that we are given a set $D$ of readers where each
reader is modeled by some disc in the plane. Let $P$ be a set of
tags (modeled as points) that lie in the union of the discs in
$D$. Suppose that all readers in $D$ use the same wireless
frequency. For the sake of simplicity, suppose also that each reader is only allowed to be activated once.
The goal is to schedule for each reader $d \in D$ a
time slot $t(d)$ for which the reader $d$ will be active. That
is, at time $t(d)$ reader $d$ would initiate a `read' action. We
further assume that a given tag $p \in P$ can be read by reader
$d \in D$ at time $t$ if $p \in d$ and $d$ is initiating a `read'
action at time $t$ (namely, $t=t(d)$) and no other reader $d'$
for which $p \in d'$ is active at time $t$. We say that $P$ is
read by our schedule, if for every $p \in P$ there is at least
one $d \in D$ and a time $t$ such that $p$ is read by $d$ at time
$t$. Obviously, we would like to minimize the total time slots
used in the schedule. Thus our goal is to find a function $t:D
\rightarrow \{1,\ldots,k\}$ which is conflict-free for the
hypergraph $H(D)$. Since we want to minimize the total time slots
used, again the question of what is the minimum number of colors
that always suffice to CF-color any hypergraph induced by a
finite set of $n$ discs is of interest.

\subsubsection{Vertex ranking}
Let $G=(V,E)$ be a simple graph. An {\em ordered coloring} (also
a {\em vertex ranking}) of $G$ is a coloring of the vertices
$\chi:V \rightarrow \{1,\ldots,k\}$ such that whenever two
vertices $u$ and $v$ have the same color $i$ then every simple
path between $u$ and $v$ contains a vertex with color greater
than $i$. Such a coloring has been studied before and has several
applications. It was studied in the context of VLSI design
\cite{Sen92} and in the context of parallel Cholesky factorization
of matrices \cite{LIU86}. The vertex ranking problem is also
interesting for the Operations Research community. It has
applications in planning efficient assembly of products in
manufacturing systems \cite{optnoderanktree}. In general, it
seems that the vertex ranking problem can model situations where
inter-related tasks have to be accomplished fast in parallel, with
some constrains (assembly from parts, parallel query optimization
in databases, etc.). See also \cite{orderedcoloring,Schaffer89}

The vertex ranking coloring is yet another special
form of UM-coloring. Given a graph $G$, consider the hypergraph
$H=(V,E')$ where a subset $V' \subseteq V$ is a hyperedge in $E'$ if
and only if $V'$ is the set of vertices in some simple path of
$G$. It is easily observed that an ordered coloring of $G$ is equivalent to a UM-coloring of $H$.

\subsection{A General Conflict-Free coloring Framework}
\seclab{general} Let $P$ be a set of $n$ points in $\Re^2$ and let
$\cal D$ be the set of all planar discs. In \cite{ELRS,SmPHD} it
was proved that $\um(H_{\D}(P)) = O(\log n)$ and that this bound is asymptotically tight
since for any $n \in \mathbb N$ there exist hypergraphs induced
by sets of $n$ points in the plane (w.r.t discs) which require
$\Omega(\log n)$ in any CF-coloring. In fact, Pach and T{\' o}th
\cite{cf1} proved a stronger lower-bound by showing that for any
set $P$ of $n$ points it holds that $\cf(H_{\D}(P)) = \Omega(\log n)$. The proofs of
\cite{ELRS,SmPHD} are algorithmic and rely on two crucial
properties: The first property is that the Delaunay graph
$G(H_{\D}(P))$ always contains a ``large" independent set. The second
is the following shrinkability property of discs: For every disc
$d$ containing a set of $i \geq 2$ points of $P$ there is another
disc $d'$ such that $d'\cap P \subseteq d \cap P$ and
$\cardin{d'\cap P} = 2$.

In \cite{ELRS, SmPHD} it was also proved
that, if $D$ is a set of $n$ discs in the plane, then $\um(H(D)) =
O(\log n)$. This bound was obtained by a reduction to a
three-dimensional problem of UM-coloring a set of $n$ points in
$\Re^3$ with respect to lower half-spaces. Later, Har-Peled and
Smorodinsky \cite{HS02} generalized this result to pseudo-discs
using a probabilistic argument. Pach and Tardos \cite{CFPT09}
provided several non-trivial upper-bounds on the CF-chromatic
number of arbitrary hypergraphs. In particular they showed that for every
hypergraph $H$ with $m$ hyperedges $$\CF(H) \leq  1/2 +\sqrt{2m +
1/4}$$

Smorodinsky \cite{smoro} introduced the following general
framework for UM-coloring any hypergraph. This framework holds
for arbitrary hypergraphs and the number of colors used is
related to the chromatic number of the underlying hypergraph.
Informally, the idea is to find a proper coloring with very `few'
colors and assign to all vertices of the largest color class the
final color `1', discard all the colored elements and recursively
continue on the remaining sub-hypergraph. See
Algorithm~\ref{CF-framework} below.

\begin{algorithm}[]
\caption{UMcolor$(H)$: {\it UM-coloring of a hypergraph
$H=(V,\E)$}.}
  \label{CF-framework}
  \begin{algorithmic}[1]
    \STATE $i\gets 0$: {\it $i$ denotes an unused color}
    \WHILE{$V\neq \emptyset$}
    \STATE{\bf Increment:} $i\gets i+1$
    \STATE {\bf Auxiliary coloring:} {find a proper coloring $\chi$ of the induced sub-hypergraph $H(V)$ with
    ``few'' colors}
%
    \STATE {\bf $V' \gets$ Largest color class of $\chi$}
    \STATE {\bf Color:} $f(x)\gets i ~,~ \forall x\in V'$
    \STATE {\bf Prune:} $V\gets V\setminus V'$
\ENDWHILE
\end{algorithmic}
\end{algorithm}

\begin{theorem}[\cite{smoro}]
\theolab{CF-framework}
 Algorithm~\ref{CF-framework} outputs a
valid UM-coloring of $H$.
\end{theorem}

\begin{proof}
Formally, Algorithm~\ref{CF-framework} is not well defined as its
output depends on the auxiliary coloring of step 4 of the
algorithm. Nevertheless, we regard step 4 as given to us by some
`black' box and we treat this aspect of the algorithm later on.
For a hyperedge $e \in \E$, let $i$ be the maximal index (color)
for which there is a vertex $v \in e$ colored with $i$. We claim
that there is exactly one such vertex. Indeed, assume to the
contrary that there is another such vertex $v' \in e$. Consider
the $ith$ iteration and let $V'$ denote the set of vertices of
$V$ that are colored with color greater or equal to $i$. Namely,
$V'$ is the set of vertices that `survived' all the prune steps
up to iteration $i$ and reached iteration $i$. Let $\chi$ denote
the auxiliary proper coloring for the hypergraph $H(V')$ in
iteration $i$. Since $e' = e \cap V'$ is a hyperedge of $H(V')$
and  $v$ and $v'$ belong to the same color class of $\chi$  and
$v,v' \in e'$ and since $\chi$ is a non-monochromatic coloring,
there must exist a third vertex $v'' \in e' $ such that
$\chi(v'') \neq \chi(v)$. This means that the final color of
$v''$ is greater than $i$, a contradiction to the maximality of
$i$ in $e$. This completes the proof of the theorem.
\end{proof}

The number of colors used by Algorithm~\ref{CF-framework} is the
number of iterations that are performed (i.e., the number of prune
steps). This number depends on the `black-box' auxiliary coloring
provided in step 4 of the algorithm. If the auxiliary coloring
$\chi$ uses a total of $C_i$ colors on $\cardin{V_i}$ vertices,
where $V_i$ is the set of input vertices at iteration $i$, then
by the pigeon-hole principle one of the colors is assigned to at
least $\frac{\cardin{V_i}}{C_i}$ vertices so in the prune step of
the same iteration at least $\frac{\cardin{V_i}}{C_i}$ vertices
are discarded. Thus, after $l$ iterations of the algorithm we are
left with at most $\cardin{V}\cdot \Pi_{i=1}^l (1-\frac{1}{C_i})$
vertices. If this number is less than $1$, then the number of
colors used by the algorithm is at most $l$. If for example
$C_i=2$ for every iteration, then the
algorithm discards at least $\frac{\cardin{V_i}}{2}$ vertices in
each iteration so the number of vertices left after $l$
iterations is at most $\cardin{V}(1-\frac{1}{2})^l$ so for $l =
\lfloor \log n \rfloor+1$ this number is less than $1$. Thus the number of
iterations is bounded by $\lfloor \log n \rfloor +1$ where $n$ is the number of
vertices of the input hypergraph. In the next section we analyze
the chromatic number $\chi(H)$ for several geometrically induced
hypergraphs and use Algorithm~\ref{CF-framework} to obtain bounds
on $\um(H)$.

We note that, as observed above, for a hypergraph $H$ that admits a
proper coloring with ``few'' colors hereditarily (that is, every induced sub-hypergraph admits a proper coloring with ``few" colors), $H$ also admits
a UM-coloring with few colors. The following theorem summarizes this fact:

\begin{theorem} [\cite{smoro}]
\label{th:reduce-1} Let $H = (V, \E)$ be a hypergraph with $n$
vertices, and let $k\in \mbb{N}$ be a fixed integer, $k \geq
2$. If every induced sub-hypergraph $H'\subseteq H$ satisfies
$\chi(H')) \leq k$, then $\um(H) \leq
\log_{1+\frac{1}{k-1}} n = O(k \log n)$.
\end{theorem}

\begin{remark}
We note that the parameter $k$ in Theorem~\ref{th:reduce-1} can be
replaced with a non-constant function $k=k(H')$. For example, if $k(H')={(n')}^{\alpha}$ where $0 < \alpha \leq 1$ is a fixed real and
$n'$ is the number of vertices of $H'$, an easy calculation shows that $\um(H) = O(n^{\alpha})$ where $n$ is the number of vertices of $H$.
\end{remark}

As we will see, for many of the
hypergraphs that are mentioned in this survey, the two numbers $\chi(H), \um(H)$ are
only a polylogarithmic (in $\cardin{V}$) factor apart. For the proof to work, the
requirement that a hypergraph $H$ admits a proper coloring with
few colors hereditarily is necessary. One example is the $3$-uniform hypergraph $H$ with $2n$ vertices given above.
We have $\chi(H) = 2$ and $\um(H)=n+1$. Obviously $H$ does not admit a proper $2$-coloring hereditarily.

\section{Conflict-Free Coloring of Geometric Hypergraphs}
\label{sec:geom_hyper}
\subsection{Discs and Pseudo-Discs in the Plane}
\subsubsection{Discs in $\Re^2$}
In \cite{smoro} it was shown that the chromatic number of a
hypergraph induced by a family of $n$ discs in the plane is
bounded by four. That is, for a finite family $D$ of $n$ discs in
the plane we have:
\begin{theorem}[\cite{smoro}]
\label{four_colors} $\chi(H(D)) \leq 4$
\end{theorem}

Combining Theorem~\ref{th:reduce-1} and
Theorem~\ref{four_colors} we obtain the following:
\begin{theorem}[\cite{smoro}]
\label{cf-discs} Let $\D$ be a set of $n$ discs in the plane.
Then $\um(H(\D)) \leq \log_{4/3} n$.
\end{theorem}

\begin{proof}
We use Algorithm~\ref{CF-framework} and the auxiliary proper four
coloring provided by Theorem~\ref{four_colors} in each prune step.
Thus in each step $i$ we discard at least $\cardin{V_i}/4$ discs
so the total number of iterations is bounded by $ \log_{4/3} n$.
\end{proof}

{\noindent \bf Remark:} The existence of a four coloring provided
in Theorem~\ref{four_colors} is algorithmic and uses the
algorithm provided in the Four-Color Theorem \cite{AH77a,AH77b}
which runs in linear time. It is easy to see that the total
running time used by algorithm~\ref{CF-framework} for this case
is therefore $O(n \log n)$. The bound in Theorem~\ref{cf-discs} holds also for the case of hypergraphs induced by points in the plane with respect to discs. This follows from the fact that such a hypergraph $H$ satisfies $\chi(H) \leq 4$. Indeed, the Delaunay graph $G(H)$ is planar (and hence four colorable) and any disc containing at least $2$ points also contains an edge of $G(H)$ \cite{ELRS}.

Smorodinsky \cite{smoro} proved that there exists an absolute
constant $C$ such that for any family $\cal P$ of pseudo-discs in
the plane $\chi(H({\cal P})) \leq C$. Hence, by Theorem~\ref{th:reduce-1} we have $\um(H({\cal P})) = O(\log n)$. It is not known what is the
exact constant and it might be possible that it is still $4$. By
taking $4$ pair-wise (openly-disjoint) touching discs, one can
verify that it is impossible to find a proper coloring of the discs
with less than $4$ colors.

There are natural geometric hypergraphs which require $n$ distinct
colors even in any proper coloring. For example, one can place a
set $P$ of $n$ points in general position in the plane (i.e., no
three points lie on a common line) and consider those ranges that
are defined by rectangles. In any proper coloring of $P$ (w.r.t
rectangles) every two such points need distinct colors since for
any two points $p,q$ there is a rectangle containing only $p$ and
$q$.

One might wonder what makes discs more special than other shapes?
Below, we show that a key property that allows CF-coloring discs
with a ``small"  number of colors unlike rectangles is the so
called ``low" {\em union-complexity} of discs.

\begin{definition}
Let $\R$ be a family of $n$ simple Jordan regions in the plane.
The {\em union complexity} of $\R$ is the number of vertices
(i.e., intersection of boundaries of pairs of regions in $\R$)
that lie on the boundary $\partial \bigcup_{r \in \R} r$.
\end{definition}

As mentioned already, families of discs or
pseudo-discs in the plane induce hypergraphs with chromatic number
bounded by some absolute constant. The proof of \cite{smoro} uses
the fact that pseudo-discs have  ``linear union
complexity'' \cite{KLPS}.


%

The following theorem bounds the chromatic number of a hypergraph
induced by a finite family of regions $\R$ in the plane as a
function of the union complexity of $\R$:
\begin{theorem} [\cite{smoro}]
    Let $\R$ be a set of $n$ simple Jordan regions and let $\UComp: \mathbb N \rightarrow \mathbb N$
    be a function such that $U(m)$ is the maximum union complexity of any $k$ regions in $\R$ over all $k \leq m$,
    for $1 \leq m \leq n$.
    We assume that $ \frac{\UComp(m)}{m}$ is a non-decreasing function.
Then, $\chi(H(\R)) = O(\frac{\UComp(n)}{n})$. Furthermore, such a
coloring can be computed in polynomial time under a proper and
reasonable model of computation.
    \label{main}
\end{theorem}

As a corollary of Theorem~\ref{main}, for any family $\R$ of $n$ planar Jordan regions for which the union-complexity function $\UComp(n)$ is linear,
we have that $\chi(H(\R)) = O(1)$. Hence, combining Theorem~\ref{main} with Theorem~\ref{th:reduce-1} we have:

\begin{theorem}[\cite{smoro}]\label{th:union-UMcoloring}
 Let $\R$ be a set of $n$ simple Jordan regions and let $\UComp: \mathbb N \rightarrow \mathbb N$
    be a function such that $U(m)$ is the maximum complexity of any $k$ regions in $\R$ over all $k \leq m$,
    for $1 \leq m \leq n$. If $\R$ has linear union complexity in the sense that $\UComp(n) \leq Cn$ for some constant $C$, then $\um(H(\R)) = O(\log n)$.
\end{theorem}

\subsection{Axis-Parallel rectangles}
\subsubsection{hypergraphs induced by axis-parallel rectangles}
As mentioned already, a hypergraph induced by $n$ rectangles in
the plane might need $n$ colors in any proper coloring. However,
in the special case of axis-parallel rectangles, one can obtain
non-trivial upper bounds. Notice that axis-parallel rectangles
might have quadratic union complexity so using the above
framework yields only the trivial upper bound of $n$.
Nevertheless, in \cite{smoro} it was shown that any hypergraph
that is induced by a family of $n$ axis-parallel rectangles,
admits an $O(\log n)$ proper coloring. This bound is
asymptotically tight as was shown recently by Pach and Tardos
\cite{PTrects}.

\begin{theorem}[\cite{smoro}]
\label{rectangles} Let $\R$ be a set of $n$ axis-parallel
rectangles in the plane. Then $\chi(H(\R)) \leq 8 \log n$.
\end{theorem}

Plugging this fact into Algorithm~\ref{CF-framework} yields:
\begin{theorem}[\cite{smoro}]
\label{cf-rectangles} Let $\R$ be a set of $n$ axis-parallel
rectangles in the plane. Then $\um(H(\R)) = O(\log^2 n)$.
\end{theorem}


{\bf \noindent Remark:} Notice that in particular there exists a
family $\R$ of $n$ axis-parallel rectangles for which $\CF(H(\R)) =
\Omega(\log n)$. Another example of a hypergraph $H$ induced by $n$ axis-parallel squares with $\chi(H)=2$ and $\CF(H)=\Omega(\log n)$ is given in Figure~\ref{fig:rectangles-dih}. This hypergraph is, in fact, isomorphic to the discrete interval hypergraph with $n$ vertices.

\begin{figure}[htbp]
    \begin{center}
      \includegraphics[width=0.3\textwidth]{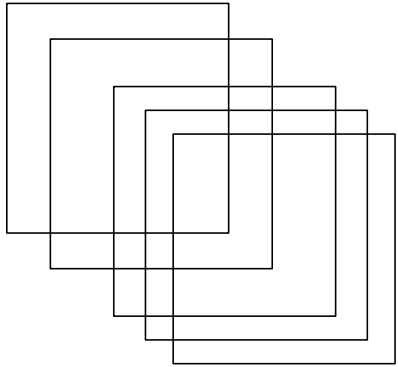}
      \caption{An example of $n$ axis-parallel squares inducing the hypergraph $H$ with $\chi(H)=2$ and $\CF(H) = \Omega(\log n)$.}
            \label{fig:rectangles-dih}
    \end{center}
  \end{figure}

\begin{problem}
Close the asymptotic gap between the best known upper bound $O(\log^2 n)$
and the lower bound $\Omega(\log n)$ on the CF-chromatic number of hypergraphs induced by $n$ axis-parallel rectangles in the plane.
\end{problem}

\subsubsection{Points with respect to axis-parallel rectangles}

Let $\R$ be the family of all axis-parallel rectangles in the
plane. For a finite set $P$ in the plane, let $H(P)$ denote the hypergraph $H_{\R}(P)$.
Let $D(P)$ denote the
Delaunay graph of $H(P)$. It is easily seen that $\chi(D(P)) =
\chi(H(P))$ since every axis-parallel rectangle containing at
least two points, also contains an edge of $D(P)$.

 The following problem
seems to be rather elusive:

\begin{problem}
Let $\R$ be the family of all axis-parallel rectangles in the
plane. Let $d=d(n)$ be the least integer such that for any set $P$
of $n$ points in the plane $\chi(D(P)) \leq d(n)$. Provide sharp
asymptotic bounds on $d(n)$.
\end{problem}

It was first observed in \cite{HS02} that $d(n) = O(\sqrt{n})$ by
a simple application of the classical Erd{\H o}s-Szekeres theorem
for a sequence of reals. This theorem states that in a sequence
of $k^2+1$ reals there is always a monotone subsequence of length
at least $k+1$ (see, e.g., \cite{w-igt-01}).

One can show that for any set $P$ of $n$ points in the plane
there is a subset $P' \subset P$ of size $\Omega(\sqrt{n})$ which
is independent in the graph $D(P)$. To see this, sort the points
$P=\{p_1,\ldots,p_n\}$ according to their $x$-coordinate. Write
the sequence of $y$-coordinates of the points in $P$
$y_1,\ldots,y_n$. By the Erd{\H o}s-Szekeres theorem, there is a
subsequence $y_{i_1},\ldots,y_{i_k}$ with $k = \Omega(\sqrt{n})$
which is monotone. We refer to the corresponding subset of $P$ as
a {\em monotone chain}. Notice that by taking every other point
in the monotone chain, the set $p_{i_1},p_{i_3},p_{i_5},\ldots$ is
a subset of size $k/2 = \Omega(\sqrt{n})$ which is independent in
$D(P)$. See Figure~\ref{erdszek} for an illustration. In order to
complete the coloring it is enough to observe that one can
iteratively partition $P$ into $O(\sqrt{n})$ independent sets of
$D(P)$.

\begin{figure}[t]
    \begin{center}
        \includegraphics[width=6cm]{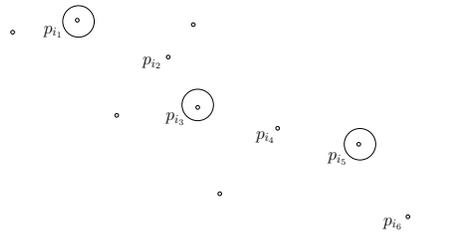}
        \caption{The circled points form an independent set in
           the Delaunay graph $D(P)$.}

        \label{erdszek}
    \end{center}
\end{figure}
The bounds on $d(n)$ were recently improved and the best known
bounds are stated below:

{\bf \noindent Upper bound: \cite{cf9}} $d(n) =
\tilde{O}(n^{0.382})$

{\bf \noindent Lower bound: \cite{cf5}} $d(n) = \Omega(\frac{\log
n}{\log^2 \log n} )$

We give a short sketch of the ideas presented in \cite{cf9} in
order to obtain the upper bound $d(n) = \tilde{O}(n^{0.382})$
where $\tilde{O}$ denotes the fact that a factor of polylog is
hiding in the big-$O$ notation. Our presentation of the ideas is
slightly different from \cite{cf4,cf9} since our aim is to bound
$d(n)$ which corresponds to coloring the Delaunay graph of $n$
points rather than CF-coloring the points themselves. However, as
mentioned above, such a bound implies also a similar bound on the
CF-chromatic number of the underlying hypergraph. Assume that $d(n) \geq c \log n$ for some fixed constant $c$. We
will show that $d(n) = O(n^{\alpha})$ for all $\alpha > \alpha_0 = \frac{3-\sqrt{5}}{2}$. The proof relies on the following
key ingredient, first proved in \cite{cf4}. For a point set $P$
in the plane, let $G_r$ be an $r\times r$ grid such that each
row of $G_r$ and each column of $G_r$ contains at most $\lceil n/r \rceil$ points of
$P$. Such a grid is easily seen to exists. A coloring of $P$ is called a {\em quasi-coloring} with
respect to $G_r$ if every rectangle that is fully contained in a
row of $G_r$ or fully contained in a column of $G_r$ is
non-monochromatic. In other words, when coloring $P$, we do not
care about rectangles that are not fully contained in a row or
fully contained in a column (or contain only one point).

\begin{lemma}[\cite{cf4,cf9}]
Let $P$ be a set of $n$ points in the plane. If $\Omega(\log n) =
d(n) = O(n^{\alpha})$ then for every $r$, $P$ admits a quasi
-coloring with respect to $G_r$ with
$\tilde{O}({(\frac{n}{r})}^{2\alpha- \alpha^2})$ colors.
\end{lemma}
The proof of the lemma uses a probabilistic argument. We first
color each column in $G_r$ independently with $d(n/r)$ colors.
Then for each column we permute the colors randomly and then
re-color all points in a given row that were assigned the same
initial color. We omit the details of the proof and its
probabilistic analysis.

Next, we choose an appropriate subset $P' \subset P$ which
consists of $O(r)$ monotone chains and with the following key property:
If a rectangle $S$ contains
points from at least two rows of $G_r$ and at least two columns
of $G_r$, then $S$ also contains a point of $P'$.
Note that a chain can be
colored with $2$ colors so altogether one can color $P'$ with
$O(r)$ colors, not to be used for $P\setminus P'$. Thus a rectangle that is not fully
contained in a row or a column of $G_r$ is served by the coloring. Hence, it is enough to quasi-color the
points of $P\setminus P'$ with respect to $G_r$. By the above
lemma, the total number of colors required for such a coloring is
$\tilde{O}({(\frac{n}{r})}^{2\alpha- \alpha^2} + r)$. Choosing
$r=n^{\frac{2\alpha- \alpha^2}{1+2\alpha- \alpha^2}}$ we obtain
the bound
$\tilde{O}(n^{\frac{2\alpha-\alpha^2}{1+2\alpha-\alpha^2}})$.
Thus, taking $\alpha_0$ to satisfy the equality
$$
\alpha_0 = \frac{2\alpha_0- {\alpha_0}^2}{1+2{\alpha_0}- {\alpha_0}^2}
$$

or $\alpha_0 = \frac{3-\sqrt{5}}{2}$, we have that for $\alpha > \alpha_0$ $d(n) = O(n^{\alpha})$ as asserted.

  \begin{figure}[htbp]
   \begin{center}
      \includegraphics[width=0.3\textwidth ]{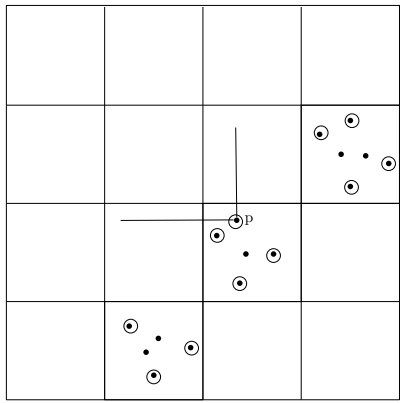}
      \caption{The grid $G_r$ (for $r = 4$) and one of its positive diagonals.
      The circled points are taken to be in $P'$ and the square
      points are in $P \setminus P'$. The point $p$ is an extreme point of type $2$ in that diagonal and is also an extreme point
      of type $1$ in the negative diagonal that contains the grid cell of $p$.}
            \label{fig:grid}
   \end{center}
  \end{figure}

To complete the proof, we need to construct the set $P'$.
Consider the diagonals of the grid $G_r$. See
Figure~\ref{fig:grid} for an illustration. In each positive
diagonal we take the subset of (extreme) points of type $2$ or
$4$, where a point $p$ is said to be of type $2$ (respectively,
$4$) if the $2$'nd quadrant (respectively, the $4$'th quadrant)
with respect to $p$ (i.e., the subset of all points above and to the left of $p$) does not contain any other point from the
diagonal. Similarly, for diagonals with negative slope we take
the points of type $1$ and $3$. If a point belongs to more than
one type (in the two diagonals that contain the point) then we
arbitrarily choose one of the colors it gets from one of the
diagonals. It is easy to see that the set $P'$ admits a
proper coloring with $O(r)$ colors, as there are only $2r-1$
positive diagonals and $2r-1$ negative diagonals, and in each
diagonal the extreme points of a fixed type form a monotone chain.

As mentioned, reducing the gap between the best known asymptotic upper and lower bounds mentioned above is a very interesting open problem.

\subsection{Shallow Regions}
As mentioned already, for every $n$ there are sets $D$ of $n$
discs in the plane such that $\CF(H(D)) = \Omega(\log n)$. For
example, one can place $n$ unit discs whose centers all lie on a
line, say the $x$-axis, such that the distance between any two
consecutive centers is less than $1/n$. It was shown in
\cite{ELRS} that, for such a family $D$, $\CF(H(D)) = \Omega(\log
n)$ since $H(D)$ is isomorphic to the discrete interval hypergraph with $n$ vertices. However, in this case there are points that are covered by
many of the discs of $D$ (in fact, by all of them). This leads to
the following fascinating problem: What happens if we have a
family of $n$ discs $D$ with the property that every point is
covered by at most $k$ discs of $D$, for some parameter $k$. It is
not hard to see that in such a case, one can color $D$ with
$O(k)$ colors such that any two intersecting discs have distinct
colors. However, we are interested only in CF-coloring of $D$.
Let us call a family of regions, with the property that no point
is covered by more than $k$ of the regions, a {\em $k$-shallow}
family.

\begin{problem}
What is the minimum integer $f=f(k)$ such that for any finite
family of $k$-shallow discs $D$, we have: $\CF(H(D)) \leq f(k)$?
\end{problem}
As mentioned already, it is easy to see that $f(k)=O(k)$.
However, it is  conjectured that the true upper bound should be
polylogarithmic in $k$.

In the further restricted case that any disc in $D$ intersects at
most $k$ other discs, Alon and Smorodinsky \cite{AS08} proved
that $\CF(H(D)) = O(\log^3 k)$ and this was recently improved by
Smorodinsky \cite{cf_shallow} to $\CF(H(D)) = O(\log^2 k)$. Both
bounds also hold for families of pseudo-discs. We sketch the
proof of the following theorem:

\begin{theorem}[\cite{cf_shallow}]
\label{shallow_main} Let $D$ be a family of $n$ discs in the plane
such that any disc in $D$ intersects at most $k$ other discs in
$D$. Then $\CF(H(D)) = O(\log^2 k)$
\end{theorem}

The proof of Theorem~\ref{shallow_main} is probabilistic and uses
the Lov{\' a}sz Local Lemma \cite{ALON00}. We start with a few
technical lemmas:

Denote by $E_{\leq \ell}(D)$ the subset of hyperedges of $H(D)$
of cardinality less than or equal to $\ell$.
\begin{lemma}
\label{at most k} Let $D$ be a finite set of $n$ planar discs.
Then $\cardin{E_{\leq k}(D)} = O(kn)$.
\end{lemma}

\begin{proof}
This easily follows from the fact that discs have linear
union-complexity \cite{KLPS} and the Clarkson-Shor probabilistic
technique \cite{cs-arscg-89}. We omit the details of the proof.
\end{proof}

\begin{lemma}
\label{shallow} Let $D$ be a set of $n$ planar discs, and let
$\ell>1$ be an integer. Then the hypergraph $(D,E_{\leq
\ell}(D))$ can be CF-colored with $O(\ell)$ colors.
\end{lemma}

{\bf \noindent Remark:} In fact, the proof of Lemma~\ref{shallow}
which can be found in \cite{aloupis} provides a stronger
coloring. The coloring has the property that every hyperedge in
$E_{\leq \ell}(D)$ is colorful (i.e., all vertices have distinct
colors). Such a coloring is referred to as $\ell$-colorful
coloring and is discussed in more details in
Subsection~\ref{k-strong}

\begin{lemma}
\label{lll}
 Let $D$ be a set of discs such that every disc
intersects at most $k$ others. Then there is a constant $C$ such
that $D$ can be colored with two colors (red and blue) and such
that for every face $f \in \Arr(D)$ with depth at least $C \ln k$,
there are at least $\frac{\cardin{d(f)}}{3}$ red discs containing
$f$ and at least $\frac{\cardin{d(f)}}{3}$ blue discs containing
$f$, where $d(f)$ is the set of all discs containing the face $f$.
\end{lemma}

\begin{proof}
Consider a random coloring of the discs in $D$, where each disc $d
\in D$ is colored independently red or blue with probability
$\frac{1}{2}$. For a face $f$ of the arrangement $\Arr(D)$ with
$\cardin{d(f)} \geq C \ln k$ (for some constant $C$ to be
determined later), let $A_f$ denote the ``bad'' event that either
less than $\frac{\cardin{d(f)}}{3}$ of the discs in $d(f)$ or
more than $\frac{2\cardin{d(f)}}{3}$ of them are colored blue. By
the Chernoff inequality (see, e.g., \cite{ALON00}) we have:
$$
Pr[A_f] \leq 2e^{-\frac{\cardin{d(f)}}{72}} \leq 2e^{-\frac{C\ln
k}{72}}
$$
We claim that for every face $f$, the event $A_f$ is mutually
independent of all but at most $O(k^3)$ other events. Indeed
$A_f$ is independent of all events $A_s$ for which $d(s) \cap
d(f) = \emptyset$. By assumption, $\cardin{d(f)} \leq k+1$.
Observe also that a disc that contains $f$, can contain at most
$O(k^2)$ other faces, simply because the arrangement of $k$ discs
consists of at most $O(k^2)$ faces. Hence, the claim follows.

Let $C$ be a constant such that:
$$
e \cdot 2e^{-\frac{C\ln k}{72}}\cdot 2k^3 < 1
$$

By the Lov{\' a}sz Local Lemma, (see, e.g., \cite{ALON00}) we
have:
$$
Pr[\bigwedge_{\cardin{d(f)} \geq C \ln k} \bar{A_f}]
> 0
$$
In particular, this means that there exists a coloring for which
every face $f$ with $\cardin{d(f)} \geq C \ln k$ has at least
 $\frac{\cardin{d(f)}}{3}$ red discs containing $f$
and at least $\frac{\cardin{d(f)}}{3}$ blue discs  containing it,
as asserted. This completes the proof of the lemma.
\end{proof}

{\bf \noindent Proof of Theorem~\ref{shallow_main}:} Consider a
coloring of $D$ by two colors as in Lemma~\ref{lll}. Let $B_1$
denote the set of discs in $D$ colored blue. We will color the
discs of $B_1$ with $O(\ln k)$ colors such that $E_{\leq 2 C\ln
k}(B_1)$ is conflict-free, as guaranteed by Lemma~\ref{shallow},
and recursively color the discs in $D\setminus B_1$ with colors
disjoint from those used to color $B_1$. This is done, again, by
splitting the discs in $D \setminus B_1$ into a set of red discs
and a set $B_2$ of blue discs with the properties guaranteed by
Lemma~\ref{lll}. We repeat this process until every face of the
arrangement $\Arr(D')$ (of the set $D'$ of all remaining discs)
has depth at most $C\ln k$. At that time, we color $D'$ with
$O(\ln k)$ colors as described in Lemma~\ref{shallow}. To see that this
coloring scheme is a valid
conflict-free coloring, consider a point $p \in \bigcup_{d \in D}
d$. Let $d(p) \subset D$ denote the subset of all discs in $D$ that contain
$p$. Let $i$ be the largest index for which $d(p) \cap B_i \neq
\emptyset$. If $i$ does not exist (namely, $d(p) \cap B_i =
\emptyset$  $\forall i$) then by Lemma~\ref{lll} $\cardin{d(p)}
\leq C\ln k$. However, this means that $d(p) \in E_{\leq C\ln k}(D)$
and thus $d(p)$ is conflict-free by the coloring of the last
step. If $\cardin{d(p) \cap B_i} \leq 2 C\ln k$ then $d(p)$ is conflict free since
one of the colors in $d(p) \cap B_i$ is unique according to the coloring of $E_{\leq c \ln k}(B_i)$. Assume then, that
$\cardin{d(p) \cap B_i} > 2C\ln k$. Let $x$ denote the number of
discs containing $p$ at step $i$. By the property of the coloring
of step $i$, we have that $x \geq 3C\ln k$. This means that after
removing $B_i$, the face containing $p$ is also contained in at
least $C\ln k$ other discs. Hence, $p$ must also belong to a disc
of $B_{i+1}$, a contradiction to the maximality of $i$. To argue
about the number of colors used by the above procedure, note that
in each prune step, the depth of every face with depth
$i \geq C \ln k$ is reduced with a factor of at least $\frac{1}{3}$. We
started with a set of discs such that the maximal depth is $k+1$.
After the first step, the maximal depth is $\frac{2}{3}k$ and for
each step we used $O(\ln  k)$ colors so, in total, we have that
the maximum number of colors $f(k,r)$, needed for CF-coloring a
family of discs with maximum depth $r$ such that each disc
intersects at most $k$ others satisfies the recursion:
$$
f(k,r) \leq O(\ln  k) + f(k,\frac{2}{3}r).
$$
This gives $f(k,r) = O(\ln k \log r)$. Since, in our case $r \leq
k+1$, we obtain the asserted upper bound. This completes the
proof of the theorem. $\Box$

{\bf \noindent Remark:} Theorem~\ref{shallow_main} works almost
verbatim for any family of regions (not necessarily convex) with
linear union complexity. Thus, for example, the result applies to
families of homothetics or more generally to any family of
pseudo-discs, since pseudo-discs
have linear union complexity (\cite{KLPS}). We also note that, as in other cases mentioned so far, it is easily seen that the proof of the bound of Theorem~\ref{shallow_main} holds for UM-coloring.

The proof of Theorem~\ref{shallow_main} is non-constructive
since it uses the Lov{\' a}sz Local Lemma. However, we can use
the recently discovered algorithmic version of the Local Lemma of
Moser and Tardos \cite{MT09} to obtain a constructive proof of
Theorem~\ref{shallow_main}.

\begin{problem}
As mentioned, the only lower bound that is known for this problem
is $\Omega(\log k)$ which is obvious from taking the lower bound
construction of \cite{ELRS} with $k$ discs. It would be
interesting to close the gap between this lower bound and the
upper bound $O(\log^2 k)$.
\end{problem}

The following is a rather challenging open problem:

\begin{problem}
Obtain a CF-coloring of discs with maximum depth $k+1$ (i.e., no
point is covered by more than $k+1$ discs) with only
polylogarithmic (in $k$) many colors. Obviously, the assumption
of this subsection that a disc can intersect at most $k$ others is
much stronger and implies maximum depth $k+1$. However, the
converse is not true. Assuming only bounded depth does not imply
the former. In bounded depth, we still might have discs
intersecting many (possibly all) other discs.
\end{problem}

\section{Extensions of CF-Coloring}
\subsection{$k$-CF coloring}
We generalize the notion of CF-coloring of a hypergraph to {\em
$k$-CF-coloring}. Informally, we think of a hyperedge as being
`served' if there is a color that appears in the hyperedge (at
least once and) at most $k$ times, for some fix prescribed
parameter $k$. For example, we will see that when the underlying
hypergraph is induced by $n$ points in $\Re^3$ with respect to
the family of all balls, there are $n$ points for which any
CF-coloring needs $n$ colors but there exists a 2-CF-coloring with
$O(\sqrt{n})$ colors (and a $k$-CF-coloring with $O(n^{1/k})$
colors for any fixed $k \ge 2$).  We also show that any hypergraph
 $(V,\E)$ with a finite
VC-dimension $c$, can be $k$-CF-colored with $O(\log
{\cardin{P}})$ colors, for a reasonably large $k$. This relaxation
of the model is applicable in the wireless scenario since the
real interference between conflicting antennas (i.e., antennas
that are assigned the same frequency and overlap in their
coverage area) is a function of the number of such antennas. This
suggests that if for any given point, there is some frequency
that is assigned to at most a ``small'' number of antennas that
cover this point, then this point can still be served using that
frequency because the interference between a small number of
antennas is low. This feature is captured by the following notion
of $k$-CF-coloring.

\begin{definition}
    {\bf $k$-CF-coloring of a hypergraph}: Let $H=(V,\E)$ be
    a hypergraph. A function $\chi: V
    \rightarrow \{1,\ldots,i\}$ is a {\em $k$-CF-coloring}
    of $H$ if for every $S \in \E$ there exists a color $j$ such that $1 \leq
    \cardin{\{v \in S| \chi(v) = j\}} \leq k$; that
    is, for every hyperedge $S \in \E$ there exists
    at least one color $j$ such that $j$ appears (at least
    once and) at most $k$ times among the colors assigned to
    vertices of $S$.
\end{definition}

Let $\kopt(H)$ denote the minimum number of colors needed for a
$k$-CF-coloring of $H$.

Note that a $1$-CF-coloring of a hypergraph $H$ is simply a
CF-coloring.

Here we modify Algorithm~\ref{CF-framework} to obtain a $k$-CF
coloring of any hypergraph. We need yet another definition of the
following relaxed version of a proper coloring:

\begin{definition}
Let $H=(V,\E)$ be a hypergraph. A coloring $\varphi$ of $H$ is
called {\em $k$-weak} if every hyperedge $e \in \E$ with
$\cardin{e} \geq k$ is non-monochromatic. That is,  for every
hyperedge $e \in \E$ with $\cardin{e} \geq k$ there exists at
least two vertices $x,y \in e$ such that $\varphi(x) \neq
\varphi(y)$.
\end{definition}

Notice that a $k$-weak coloring (for $k \geq 2$) of a hypergraph $H=(V,\E)$ is simply a
proper coloring for the hypergraph $(V,\E_{\geq k})$ where $\E_{\geq k}$ is the subset of hyperedges in $\E$
with cardinality at least $k$. This notion was used implicitly in \cite{HS02,smoro} and then was explicitly defined
and studied in the Ph.D. of Keszegh \cite{Kes1,Kes2}. It is also related to the notion of
cover-decomposability and polychromatic colorings (see, e.g., \cite{gibsonvar,pach,PachToth}).

 We are ready to generalize
Algorithm~\ref{CF-framework}. See Algorithm~\ref{KCF-framework}
below.

\begin{algorithm}[htb!]
\caption{$k$-CFcolor$(H)$: {\it $k$-Conflict-Free-color a
hypergraph $H=(V,\E)$}.}
  \label{KCF-framework}
  \begin{algorithmic}[1]
    \STATE $i\gets 0$: {\it $i$ denotes an unused color}
    \WHILE{$V\neq \emptyset$}
    \STATE{\bf Increment:} $i\gets i+1$
    \STATE {\bf Auxiliary coloring:} {find a weak $k+1$-coloring $\chi$ of $H(V)$ with
    ``few'' colors}
%
    \STATE {\bf $V' \gets$ Largest color class of $\chi$}
    \STATE {\bf Color:} $f(x)\gets i ~,~ \forall x\in V'$
    \STATE {\bf Prune:} $V\gets V\setminus V'$, $H\gets H(V)$
\ENDWHILE
\end{algorithmic}
\end{algorithm}

\begin{theorem}[\cite{HS02}]
Algorithm~\ref{KCF-framework} outputs a valid $k$-CF-coloring of
$H$.
\end{theorem}

\begin{proof}
The proof is similar to the proof provided in \secref{general}
for the validity of Algorithm~\ref{CF-framework}. In fact, again,
the coloring provided by Algorithm~\ref{KCF-framework} has the
stronger property that for any hyperedge $S \in \E$ the maximal
color appears at most $k$ times.
\end{proof}

As a corollary similar to the one mentioned in Theorem~\ref{th:reduce-1}, for a hypergraph $H$ that admit a
 $k+1$-weak coloring with ``few'' colors hereditarily, $H$ also admits
a $k$-CF-coloring with few colors. The following theorem summarizes this fact:

\begin{theorem} [\cite{HS02}]
\label{th:reduce-2} Let $H = (V, \E)$ be a hypergraph with $n$
vertices, and let $l,k \in \mbb{N}$ be two fixed integers, $k \geq
2$. Assume that every induced sub-hypergraph $H'\subseteq H$ admits a $k+1$-weak coloring with at most $l$ colors.
Then $H$ admits a $k$-CF-coloring with at most
$\log_{1+\frac{1}{l-1}} n = O(l \log n)$ colors.
\end{theorem}

\begin{proof}
The proof is similar to the proof of Theorem~\ref{th:reduce-1}
\end{proof}

\subsubsection{CF-Coloring of Balls in Three
   Dimensions}

\begin{lemma}
    Let $\B$ be the set of balls in three dimensions. There exists a hypergraph $H$ induced
     by a finite set $P$
    of $n$ points in $\Re^3$ with respect to $\B$ such that
    $\chi_{1CF}(H)= n$.
    The same holds for the set $\cal H$ of halfspaces in
    $\Re^d$, for $d>3$.

    \lemlab{balls:3d}
\end{lemma}

\begin{proof}
    Take $P$ to be a set of $n$ points on the positive
    portion of the moment curve
    $\gamma=\{ (t, t^2, t^3)| t \geq 0\}$ in $\Re^3$.
    It is
    easy to verify that any pair of points $p,q \in P$ are
    connected in the Delaunay triangulation of $P$
   implying that there exists a ball
    whose intersection with $P$ is $\{p,q\}$.  Thus,
    all points must be colored
    using different colors.

    The second claim follows by taking $P$ to be $n$ distinct points on the moment curve $\{(t, t^2, \ldots, t^d)\}$ in $\Re^d$ (i.e, $P$ is the set of vertices of a so-called {\em cyclic-polytope} $C(n,d)$. See, e.g., \cite{STANELY}).
\end{proof}

\begin{theorem}[\cite{HS02,SmPHD}]
    Let $P$ be a set of $n$ points in $\Re^3$. Put $H=H_{\B}(P)$.
    Then
    $\kopt(H)=O(n^{1/k})$, for any fixed constant $k
    \geq 1$.

\end{theorem}

\begin{proof}
As is easily seen by Algorithm~\ref{KCF-framework}, it is enough
to prove that $H$ admits a $k+1$-weak coloring with $O(n^{1/k})$
colors. If so, then in every iteration we discard at least
$\Omega({\cardin{P_i}}^{1-\frac{1}{k}})$ elements so the total
number of iterations (colors) used is $O(n^{1/k})$. The proof
that $H$ admits a $k+1$-weak coloring with $O(n^{1/k})$ colors
uses the probabilistic method. We provide only a brief sketch of
the proof. It is enough to consider all balls containing exactly
$k+1$ points since if a ball contains more than $k+1$ points then
by perturbation and shrinking arguments it will also contain a
subset of $k+1$ points that can be cut-off by a ball. So we may
assume that in the underlying hypergraph $H=(P,\E)$, all
hyperedges have cardinality $k+1$ (such a hypergraph is also
called a $k+1$-uniform hypergraph). So we want to color the set
$P$ with $O(n^{1/k})$ colors such that any hyperedge in $\E$ is
non-monochromatic. By the Clarkson-Shor technique, it is easy to
see that the number of hyperedges in $\E$ is $O(k^2 n^2)$. Thus the
average degree of a vertex in $H$ is $O(n)$ where the constant of
proportionality depends on $k$. It is well known that such a
hypergraph has chromatic number $O(n^{1/k})$. This is proved via
the probabilistic method. The main ingredient is the Lov{\' a}sz
Local Lemma  (see, e.g., \cite{ALON00}).
\end{proof}

In a similar way we have:
\begin{theorem}[\cite{HS02,SmPHD}]
Let $\R$ be a set of $n$ balls in $\Re^3$. Then
$\kopt(H(\R))=O(n^{1/k})$.
\end{theorem}

\subsubsection{VC-dimension and $k$-CF coloring}
\begin{definition}
    Let $H = (V,\E$) be a hypergraph. The {\em Vapnik-Chervonenkis}
    dimension (or {\em VC-dimension}) of $H$, denoted by $VC(H)$, is
    the maximal cardinality of a subset $V'\subset V$ such that
    ${\{V'\cap r| r \in \E }\} = 2^{V'}$ (such a subset is
    said to be {\em shattered}). If there are arbitrarily large
    shattered subsets in $V$ then $VC(H)$ is defined to be $\infty$.
    See \cite{MATOUSEK} for discussion of VC-dimension and
    its applications.
\end{definition}
There are many hypergraphs with finite VC-dimension that arise
naturally in combinatorial and computational geometry. One such
example is the hypergraph $H = (\Re^d,{\cal
   H}_d)$, where ${\cal H}_d$ is the family of all (open)
halfspaces in $\Re^d$. Any set of $d+1$ affinely independent
points is shattered in this space, and, by Radon's theorem, no
set of $d+2$ points is shattered.  Therefore $VC(H)=d+1$.

\begin{definition}
Let $(V,\E)$ be a hypergraph with $|V|=n$ and let $0 < \epsilon
\leq 1$. A subset $N \subset V$ is called an {\em $\epsilon$-net}
for $(V,\E)$ if for every hyperedge $S \in \E$ with $\cardin{S}
\geq \epsilon n$ we have $S\cap N \neq \emptyset$.
\end{definition}

Thus, an $\epsilon$-net is a {\em hitting set} of all `heavy'
hyperedges, namely, those containing at least $\epsilon n$
vertices.

An important consequence of the finiteness of the VC-dimension is
the existence of small $\epsilon$-nets, as shown by Haussler and
Welzl in \cite{hw-ensrq-87}, where the notion of VC-dimension of a
hypergraph was introduced to computational geometry.

\begin{theorem}[\cite{hw-ensrq-87}]
    For any hypergraph $H=(V,\E)$ with finite VC-dimension
    $d$ and for any $\epsilon > 0$, there exists an
    $\epsilon$-net $N \subset V$ of size
    $O(\frac{d}{\epsilon}\log{\frac{d}{\epsilon}})$.
    \theolab{VC-epsilon}
\end{theorem}

{\bf \noindent Remark:} In fact, \theoref{VC-epsilon} is valid
also in the case where $H$ is equipped with an arbitrary
probability measure $\mu$. An $\epsilon$-net in this case is a
subset $N \subset V$ that meets all hyperedges with measure at
least $\epsilon$.

Since all hypergraphs mentioned so far have finite VC-dimension,
and since some of them sometimes must be CF-colored with $n$
colors, there is no direct relationship between a finite
VC-dimension of a hypergraph and the existence of a CF-coloring of
that hypergraph with a small number of colors. In this subsection
we show that such a relationship does exist, if we are interested
in $k$-CF-coloring with a reasonably large $k$.

We first introduce a variant of the general framework for
$k$-CF-coloring of a hypergraph $H = (V,\E)$. In this framework we
modify lines $4$ and $5$ in Algorithm~\ref{KCF-framework}. In
Algorithm~\ref{KCF-framework} we first find a $k+1$-weak coloring
of the underlying hypergraph (line $4$) which is a partition of
the vertices into sets such that each set has the following
property: Every set in the partition cannot fully contain a
hyperedge with cardinality at least $k+1$. Equivalently, every
color class $V' \subset V$ has the property that every hyperedge
containing at least $k+1$ vertices of $V'$ also contain vertices
of $V\setminus V'$. We modify that framework by directly finding
a ``large" such subset in the hypergraph.

\begin{definition}
Let $H=(V,\E)$ be a hypergraph.
    A subset $V' \subset V$ is {\em $k$-admissible} if for any hyperedge $S \in \E$ with
    $\cardin{S \cap V'} > k$ we have $S\cap (V\setminus V')
    \neq \emptyset$.
\end{definition}

Assume that we are given an algorithm $\Alg$ that computes, for any
hypergraph $H=(V,\E)$, a non-empty $k$-admissible set
$V'=\Alg(H)$. We can now use algorithm $\Alg$ to $k$-CF-color the
given hypergraph (i) Compute a $k+1$-admissible set $V' =
\Alg(H)$, and assign to all the elements in $V'$ the color $1$.
(ii) Color the remaining elements in $V \setminus V'$
recursively, where in the $i$th stage we assign the color $i$ to
the vertices in the resulting $k+1$-admissible set. We denote the
resulting coloring by $C_A(H)$.

The proof of the following theorem is, yet, again, similar to that of
\theoref{CF-framework}.

\begin{theorem}[\cite{HS02,SmPHD}]
    Given a hypergraph $H=(V,\E)$, the coloring $C_A(H)$
    is a valid $k$-CF coloring of $S$.

    \theolab{kcf-range-space:layers}
\end{theorem}

\begin{lemma}
    Let $H=(V,\E)$ with $\cardin{V} = n$ be a hypergraph
with VC-dimension $d$. For any $k \geq d$ there
    exists a $k$-admissible set $V' \subset V$ with respect
    to $H$ of size $\Omega \pth{ n^{1-(d-1)/k} }$.

    \lemlab{kcf:admissible:general}
\end{lemma}

\begin{proof}
    Any coloring of $V$ is valid as far as the small hyperedges
    of $\E$ are concerned; namely, those are the hyperedges that
    contain at most $k$ vertices. Thus, let $\E'$ be the subset
    of hyperedges of $\E$ of size larger than $k$. By Sauer's
    Lemma (see, e.g., \cite{ALON00}) we have that $\cardin{\E'} \leq \cardin{\E} \leq
    n^d$.

    Next, we randomly color $V$ by black and white, where an
    element is being colored in black with probability $p$,
    where $p$ would be specified shortly. Let $I$ be the set
    of points of $V$ colored in black. If a hyperedge $r \in
    \E'$ is colored only in black, we remove one of the vertices of
    $r$ from $I$. Let $I'$ be the resulting set. Clearly,
    $I'$ is a $k$-admissible set for $H$.

    Furthermore, by linearity of expectation, the expected
    size of $I'$ is at least
    \[
    p n - \sum_{r \in \E'} p^{\cardin{r}} \geq p n - \sum_{r
       \in \E'} p^{k+1} \geq p n - p^{k+1} n^{d}.
    \]
    Setting $p = \pth{(k+1) n^{d-1}}^{-1/k}$, we have that
    the expected size of $I'$ is at least $p n - p^{k+1}
    n^{d} = p n(1-1/(k+1)) = \Omega \pth{ n^{1-(d-1)/k} }$,
    as required.
\end{proof}

As was already seen, for geometric hypergraphs one might be able
to get better bounds than the one guaranteed by
\lemref{kcf:admissible:general}.

\begin{theorem}[\cite{HS02,SmPHD}]
    Let $H=(V,\E)$ with $\cardin{V} = n$ be a finite hypergraph
     with VC-dimension $d$. Then for $k \geq d\log{n}$
    there exists a $k$-CFcoloring of $H$ with $O(\log{n})$
    colors.

    \theolab{vc-dim}
\end{theorem}

\begin{proof}
    By \lemref{kcf:admissible:general} the hypergraph $H$
    contains a $k$-admissible set of size at least $n/2$.
    Plugging this fact to the algorithm suggested by
    \theoref{kcf-range-space:layers} completes the proof of
    the theorem.
\end{proof}

As remarked above, \theoref{vc-dim} applies to all hypergraphs
mentioned in this paper. Note also, that
\lemref{kcf:admissible:general} gives us a trade off between the
number of colors and the threshold size of the coloring. As such,
the bound of \theoref{vc-dim} is just one of a family of such
bounds implied by \lemref{kcf:admissible:general}.

\subsection{$k$-Strong CF-Coloring}\label{k-strong}
Here, we focus on the notion of \emph{$k$-strong-conflict-free}
(abbreviated, $kSCF$) which is yet another extension of the
notion of CF-coloring of hypergraphs.

\begin{definition}[$k$-strong conflict-free coloring:]
Let $H= (V, \E)$ be a hypergraph and let  $k \in \mbb{N}$ be some
fixed integer. A coloring of $V$ is called {\em
$k$-strong-conflict-free} for $H$ ($k$SCF for short) if for every hyperedge $e \in
\E$ with $|e| \geq k$ there exists at least $k$ vertices in $e$,
whose colors are unique among the colors assigned to the vertices
of $e$ and for each hyperedge $e \in \E$ with $|e| < k$ all
vertices in $e$ get distinct colors. Let $f_H(k)$ denote the
least integer $l$ such that $H$ admits a $kSCF$-coloring with $l$
colors.
\end{definition}

Abellanas et al. \cite{HURTADO} were the first to study
$k$SCF-coloring\footnote{They referred to such a coloring as
$k$-conflict-free coloring.}. They focused on the special case of
hypergraphs induced by $n$ points in $\Re^2$ with respect to
discs. They showed that in this case the hypergraph admits a
$k$SCF-coloring with $O(\frac{\log n}{\log {\frac{ck}{ck-1}}})$
($ = O(k \log n)$) colors, for some absolute constant $c$.

The following notion was recently introduced and studied by
Aloupis et al. \cite{aloupis} for the special case of hypergraphs
induced by discs:

 \begin{definition}[$k$-colorful coloring]
\label{def:colorful} Let $H= (V, \E)$ be a hypergraph, and let
$\varphi$ be a coloring of $H$. A hyperedge $e \in \E$ is said to
be \emph{$k$-colorful} with respect to $\varphi$ if there exist
$k$ vertices in $e$ that are colored distinctively under
$\varphi$. The coloring $\varphi$ is called {\em $k$-colorful} if
every hyperedge $e \in \E$ is $\min \{|e|,k\}$-colorful. Let
$c_H(k)$ denote the least integer $l$ such that $H$ admits a
$k$-colorful coloring with $l$ colors.
\end{definition}

Aloupis et al. \cite{aloupis} introduced this notion explicitly and were motivated by a problem related to battery
lifetime in sensor networks. This notion is also related to the notion of polychromatic colorings. In polychromatic colorings, the general question is to estimate the minimum number $f=f(k)$ such that one can $k$-color the hypergraph with the property that all hyperedges of cardinality at least $f(k)$ are colorful in the sense that they contain a representative color of each color class.
(see, e.g., \cite{gibsonvar,efrat,PachToth} for additional details on
the motivation and related problems).

{\bf \noindent Remark:} Every $kSCF$-coloring of a hypergraph $H$
is a $k$-colorful coloring of $H$. However, the opposite claim is
not necessarily true.

The following connection between $k$-colorful coloring and
strong-conflict-free coloring of hypergraphs was proved by Horev
et al. in \cite{hks09}. If a hypergraph $H$ admits a $k$-colorful
coloring with a ``small" number of colors (hereditarily) then it
also admits a $(k-1)$SCF-coloring with a ``small" number of
colors. This connection is analogous to the connection between
non-monochromatic coloring and CF-coloring as appear in
Theorem~\ref{th:reduce-1} and the connection between $k+1$-weak coloring and $k$-CF-coloring as appear in Theorem~\ref{th:reduce-2}. We start by introducing the general
framework of \cite{hks09} for $k$SCF-coloring a given hypergraph.

\paragraph{A Framework For Strong-Conflict-Free Coloring}

Let $H$ be a hypergraph with $n$ vertices and let $k$ and $l$ be
some fixed integers such that $H$ admits the hereditary property
that every vertex-induced sub-hypergraph $H'$ of $H$ admits a
$k$-colorful coloring with at most $l$ colors. Then $H$ admits a
$(k-1)SCF$-coloring with $O(l \log n)$ colors. For the case when
$l$ is replaced with the function $k{n(H')}^{\alpha}$ we get a
better bound without the $\log n$ factor. The proof is
constructive. The following framework (denoted as Algorithm~\ref{kscf_framework}) produces a valid $(k-1)$SCF
coloring for a hypergraph $H$.

\begin{algorithm}[h!]
 \caption{(k-1)SCF-color$(H)$: {\it
$(k-1)$-Strong Conflict-Free-color a hypergraph $H=(V,\E)$}.}
\label{kscf_framework}
\begin{algorithmic}[1]
    \STATE $i \leftarrow 1$ {\it $i$ denotes an unused color}
    \WHILE{$V \not= \emptyset$}
         \STATE{\bf Increment:} $i\gets i+1$
        \STATE \bfm{Auxiliary Coloring:} {find a $k$-colorful coloring $\varphi$ of $H(V)$ with
    ``few'' colors}
    \STATE {\bf $V' \gets$ Largest color class of $\varphi$}
    \STATE {\bf Color:} $\chi(x)\gets i ~,~ \forall x\in V'$
        \STATE \bfm{Prune:} $V \leftarrow V \setminus V'$.
        \STATE \bfm{Increment:} $i \leftarrow i+1$.
    \ENDWHILE
    \STATE \bfm{Return} $\chi$.
\end{algorithmic}
\end{algorithm}

Note that Algorithm~\ref{kscf_framework} is a generalization of
Algorithm~\ref{CF-framework}. Indeed for $k=2$ the two algorithms
become identical since a $2$-colorful coloring is equivalent to a
proper coloring. Arguing about the number of colors used by the algorithm is identical to
the arguments as in the coloring produced by Algorithm~\ref{CF-framework}. The proof
or correctness is slightly more subtle.

For a hypergraph $H=(V,\E)$, we write $n(H)$ to denote the number
of vertices of $H$. As a corollary of the framework described in
Algorithm~\ref{kscf_framework} we obtain the following theorems:

\begin{theorem} [\cite{hks09}]
\label{th:reduce} Let $H = (V, \E)$ be a hypergraph with $n$
vertices, and let $k,\ell \in \mbb{N}$ be fixed integers, $k \geq
2$. If every induced sub-hypergraph $H'\subseteq H$ satisfies
$c_{H'}(k) \leq \ell$, then $f_H(k-1) \leq
\log_{1+\frac{1}{\ell-1}} n = O(l \log n)$.
\end{theorem}

\begin{theorem} [\cite{hks09}]
\label{th:nonlinear} Let $H = (V, \E)$ be a hypergraph with $n$
vertices, and let $k \geq 2$ be a fixed integer. let $0 < \alpha
\leq 1$ be a fixed real. If every induced sub-hypergraph
$H'\subseteq H$ satisfies $c_{H'}(k) = O(k{n(H')}^{\alpha})$,
then $f_H(k-1) = O(k{n}^{\alpha})$.
\end{theorem}

As a corollary of Theorem~\ref{th:reduce} and a result of Aloupis et al. \cite{aloupis} on $k$-colorful coloring of discs or points with respect to discs we obtain the following:

\begin{theorem} [\cite{hks09}]
\label{th:kscf-discs} If $H$ is a hypergraph induced by $n$
discs in the plane or a hypergraph induced by $n$ points in the plane with respect to discs then
$f_H(k) = O(k \log n)$.
\end{theorem}

\begin{proof}
The proof follows by combining the fact that $c_H(k) = O(k)$ \cite{aloupis} with Theorem~\ref{th:reduce}
\end{proof}

Theorem~\ref{th:results} below provides an upper
bound on the number of colors required by $kSCF$-coloring of
geometrically induced hypergraphs as a function of the
union-complexity of the regions that induce the hypergraphs.

Recall that, for a set $\R$ of $n$ simple closed planar Jordan regions,
$\mcal{U}_{\R}:\mbb{N} \rightarrow \mbb{N}$ is the function
defined in Theorem~\ref{main}.

\begin{theorem}[\cite{hks09}]
\label{th:prior-res} Let $k \geq 2$, let $0 \leq \alpha \leq 1$,
and let $c$ be a fixed constant. Let $\R$ be a set of $n$ simple
closed Jordan regions such that $\mcal{U}_{\R}(m) \leq c
m^{1+\alpha}$, for $1 \leq m \leq n$, and let $H=H(\R)$. Then
$c_H(k) = O( k n^{\alpha} )$.
\end{theorem}

Combining Theorem~\ref{th:reduce} with Theorem~\ref{th:prior-res}
(for $\alpha = 0$) and Theorem~\ref{th:nonlinear} with
Theorem~\ref{th:prior-res} (for $0 < \alpha < 1$) yields the
following result:

\begin{theorem}[\cite{hks09}]
\label{th:results} Let $k \geq 2$, let $0 \leq \alpha \leq 1$,
and let $c$ be a constant. Let $\R$ be a set of $n$ simple closed
Jordan regions such that $\mcal{U}_{\R}(m)= c m^{1+\alpha}$, for
$1 \leq m \leq n$. Let $H=H(\R)$. Then:
$$
 f_H(k-1) = \left\{ \begin{array}{ll}
                   O(k \log n), \mbox{      $\alpha = 0$}, \\
                   O(kn^{\alpha}), \mbox{       $0 < \alpha \leq 1$}.
                   \end{array}
            \right.
$$
\end{theorem}

{\noindent \bf Axis-parallel rectangles:} Consider
$kSCF$-colorings of hypergraphs induced by axis-parallel
rectangles in the plane. As mentioned before, axis-parallel rectangles
might have quadratic union-complexity. For a hypergraph $H$ induced by axis-parallel
rectangles, Theorem~\ref{th:results} states that $f_H(k-1) =O(k
n)$. This bound is meaningless, since the bound $f_H(k-1) \leq n$
is trivial. Nevertheless, the following theorem provides a
better upper bound for this case:

\begin{theorem}[\cite{hks09}]
\label{th:rects} Let $k \geq 2$. Let $\R$ be a set of $n$
axis-parallel rectangles, and let $H=H(\R)$. Then $f_H(k-1) =O(k
\log^2 n)$.
\end{theorem}

In order to obtain Theorem~\ref{th:rects} we need the following
theorem:

\begin{theorem} [\cite{hks09}]
\label{th:colorful_rects}
 Let $H = H(\R)$, be the
hypergraph induced by a family $\R$ of $n$ axis-parallel
rectangles in the plane, and let $k \in \mbb{N}$ be an integer,
$k \geq 2$. For every induced sub-hypergraph $H'\subseteq H$ we
have: $c_{H'}(k) \leq k \log n$.
\end{theorem}

The proof of Theorem~\ref{th:rects} is therefore an easy
consequence of Theorem~\ref{th:colorful_rects} combined with
Theorem~\ref{th:reduce}

 Har-Peled and Smorodinsky \cite{HS02} proved that any
family $\R$ of $n$ axis-parallel rectangles admit a CF-coloring
with $O(\log ^2 n)$ colors. Their proof uses the probabilistic
method. They also provide a randomized algorithm for obtaining
CF-coloring with at most $O(\log ^2 n)$ colors. Later,
Smorodinsky \cite{smoro} provided a deterministic polynomial-time
algorithm that produces a CF-coloring for $n$ axis-parallel
rectangles with $O(\log^2 n)$ colors. Theorem~\ref{th:rects} thus
generalizes the results of \cite{HS02} and \cite{smoro}. The
upper bound provided in Theorem~\ref{th:results} for $\alpha =
0$ is optimal. Specifically, there
exist matching lower bounds on the number of colors
required by any $kSCF$-coloring of hypergraphs induced by (unit)
discs in the plane.

\begin{theorem}[\cite{DeBerg}]
\label{th:lower_bounds}

(i) There exist families $\R$ of $n$ (unit) discs for which
$f_{H(\R)}(k) = \Omega(k \log n)$

(ii) There exist families $\R$ of $n$ axis-parallel squares for
which $f_{H(\R)}(k) = \Omega(k \log n)$.
\end{theorem}

Notice that for axis-parallel rectangles there is a logarithmic gap between the best known upper and lower bounds.

Theorems~\ref{th:reduce} and \ref{th:nonlinear} asserts that in
order to attain upper bounds on $f_H(k)$, for a hypergraph $H$,
one may concentrate on attaining a bound on $c_H(k)$. Given a
$k$-colorful coloring of $H$, Algorithm~\ref{kscf_framework} obtains a
strong-conflict-free coloring of $H$ in a constructive manner.
Here computational efficiency is not of main interest. However,
it can be seen that for certain families of geometrically induced
hypergraphs, Algorithm~\ref{kscf_framework} is efficient. In
particular, for hypergraphs induced by discs or axis-parallel
rectangles, Algorithm~\ref{kscf_framework} has a low
degree polynomial running time. Colorful-colorings of such
hypergraphs can be computed once the arrangement of the discs is
computed together with the depth of every face.

\subsection{List Colorings}
In view of the motivation for CF-coloring in the context of
wireless antennae, it is natural to assume that each antenna is
restricted to use some subset of the spectrum of frequencies and
that different antennae might have different such subsets
associated with them (depending, for example, on the physical
location of the antenna). Thus, it makes sense to study the
following more restrictive notion of coloring:

Let $H=(V,\E)$ with $V=\{v_1,\ldots,v_n\}$ be a hypergraph and let
$\L = \{L_1,\ldots,L_n\}$ be a family of subsets of the integers.
We say that $H$ admits a proper coloring from $\L$ (respectively,
a CF-coloring from $\L$, a UM-coloring from $\L$) if there exists a proper coloring
(respectively a CF-coloring, a UM-coloring) $C \colon  V \rightarrow \mathbb N$
such that $C(v_i) \in L_i$ for $i=1,\ldots,n$.

\begin{definition}
We say that a hypergraph $H=(V,\E)$ is {\em $k$-choosable}
(respectively, {\em $k$-CF-choosable, $k$-UM-choosable}) if for every family
$\L=\{L_1,\ldots,L_n\}$ such that $\cardin{L_i}\geq k$ for
$i=1,\ldots,n$, $H$ admits a proper-coloring (respectively a
CF-coloring, a UM-coloring) from $\L$.
\end{definition}

We are interested in the minimum number $k$ for which a given
hypergraph is $k$-choosable (respectively, $k$-CF-choosable, $k$-UM-choosable). We
refer to this number as the {choice-number} (respectively the
{\em CF-choice-number, UM-choice-number}) of $H$ and denote it by $\ch(H)$
(respectively $\cfch(H), \umch(H)$). Obviously, if the choice-number
(respectively, the CF-choice-number, UM-choice-number) of $H$ is $k$ then it can be
properly colored (respectively CF-colored, UM-colored) with at most $k$
colors, as one can proper color (respectively, CF-color, UM-color) $H$
from $\L=\{L_1,\ldots,L_n\}$ where for every $i$ we have $L_i =
\{1,\ldots,k\}$. Thus,

$$
\ch(H) \geq \chi(H) $$
$$
\cfch(H) \geq \CF(H).
$$
$$
\chum(H) \geq \um(H)
$$

Hence, any lower bound on the number of colors required by a proper
coloring of $H$ (respectively, a CF-coloring, a UM-coloring of $H$) is also a lower bound on the choice number (respectively, the CF-choice-number, the UM-choice-number) of $H$.

The study of choice numbers in the special case of graphs was
initiated by Vizing \cite{vizing} and by Erd\H{o}s Rubin and
Taylor \cite{ert}. The study of the CF-choice number and the UM-choice number of hypergraphs
was initiated very recently by Cheilaris, Smorodinsky and Sulovsk{\' y}
\cite{cf-choice}.

Let us return to the discrete interval
hypergraph $H_n$ with $n$ vertices, which was described in the
introduction. As was shown already, we have $\CF(H_n)= \um(H_n) = \lfloor
\log_2 n \rfloor + 1$. In particular we have the lower bound
$\cfch(H_n) \geq \lfloor \log_2 n \rfloor + 1$. Hence, the
following upper-bound is tight:

\begin{proposition}\label{proposition:intervalchcf}
For $n \geq 1$, $\cfch(H_n) \leq \lfloor \log_2 n \rfloor +1 $.
\end{proposition}

\begin{proof}
Assume, without loss of generality, that $n=2^{k+1} -1$. We will show
that $H_n$ is $k+1$ CF-choosable. The proof is by induction on
$k$. Let $\L = \{L_i\}_{i \in [n]}$, such that $\cardin{L_i}=k+1$,
for every $i$. Consider the median vertex $p = 2^k$.
Choose a color $x \in L_p$ and assign it to $p$. Remove $x$ from
all other lists (for lists containing $x$), i.e., consider $\L' =
\{L'_i\}_{i \in [n] \setminus p}$ where $L'_i =
L_i\setminus\{x\}$. Note that all lists in $\L'$ have size at
least $k$. The induction hypothesis is that we can CF-color any
set of points of size $2^k-1$ from lists of size $k$.
Indeed, the number of vertices smaller (respectively, larger)
than $p$ is exactly $2^k -1$. Thus, we
CF-color vertices smaller than $p$ and independently vertices
larger than $p$, both using colors from the lists of $\L'$.
Intervals that contain the median vertex $p$ also have the
conflict-free property, because
color $x$ is used only in $p$.
This completes the induction step and hence the
proof of the proposition.
\end{proof}

Note that, even in the discrete interval hypergraph, it is a more difficult problem to obtain any non-trivial upper bound on the
UM-choice number. A divide and conquer approach, along the lines of the proof of
Proposition~\ref{proposition:intervalchcf} is doomed to fail.
In such an approach, some vertex close to the median must be found,
a color must be assigned to it from its list, and this color must be
deleted from all other lists.
However, vertices close to the median might have only ``low''
colors in their lists. Thus, while we are guaranteed
that a vertex close to the median is uniquely colored for intervals containing it,
such a unique color is not necessarily the maximal color for such intervals.

Instead, Cheilaris et al. used a different approach. This approach provides a general framework for UM-coloring hypergraphs from lists. Moreover, when applied to many
geometric hypergraphs, it provides
asymptotically tight bounds for the UM-choice number.

Below, we give an informal description of that approach, which is then
summarized in Algorithm~\ref{algo:umchoice_general}. It is similar in spirit to Algorithm~\ref{CF-framework}.

Start by sorting the colors in the union of all lists in increasing order.
Let $c$ denote the minimum color.
Let $V^c \subseteq V$ denote the subset of vertices containing $c$ in their lists.
Note that $V^c$ might contain very few vertices, in fact, it might be that $\cardin{V^c} = 1$.
We simultaneously color
a suitable subset $U \subseteq V^c$ of vertices in $V^c$ with $c$.
We make sure that $U$
is independent in the sub-hypergraph $H(V^c)$.
The exact way in which we choose $U$ is crucial to the performance of the algorithm and is discussed below.
Next, for the uncolored vertices in $V^c \setminus U$, we remove the color $c$ from their lists.
This is repeated for every color in the union $\bigcup_{v \in V}L_v$ in increasing order of the colors. The algorithm stops
when all vertices are colored. Notice that such an algorithm might run into a problem, when all colors
in the list of some vertex are removed before this vertex is colored. Later, we show that if we choose
the subset $U \subseteq V^c$ in a clever way and the lists are sufficiently large, then we avoid such a problem.

\begin{algorithm}[htb!]
\caption{UMColorGeneric($H$, $\L$): Unique-maximum color hypergraph
         $H = (V, \E)$ from lists of family $\L$}
\label{algo:umchoice_general}
\begin{algorithmic}[1]
  \WHILE{$V \neq \emptyset$}
  \STATE $c \leftarrow \min \bigcup_{v \in V} L_v$
  \COMMENT{$c$ is the minimum color in the union of the lists}
  \STATE $V^c \leftarrow \{v \in V \mid c \in L_v\}$
  \COMMENT{$V^c$ is the subset of remaining vertices containing $c$ in their lists}
  \STATE $U \leftarrow$ a ``good'' independent subset of
                        the induced sub-hypergraph $H(V^c)$
  \FOR[for every vertex in the independent set,]{$x\in U$}
  \STATE $f(x) \leftarrow c$
  \COMMENT{color it with color $c$}
  \ENDFOR
  \FOR[for every uncolored vertex in $V^c$,]
      {$v \in V^c\setminus U$}
  \STATE $L_v \leftarrow L_v \setminus \{c\}$
     \COMMENT{remove $c$ from its list}
  \ENDFOR
  \STATE $V \leftarrow V \setminus U$
  \COMMENT{remove the colored vertices}
  \ENDWHILE
      \STATE \bfm{Return} $f$.
\end{algorithmic}
\end{algorithm}

As mentioned,
Algorithm \ref{algo:umchoice_general}
might cause some lists to run out of colors before coloring
all vertices. However, if this does not happen, it is
proved that the algorithm produces a UM-coloring.

\begin{lemma}\cite{cf-choice}
Provided that the lists associated with the vertices do not run out of colors during the execution of Algorithm~\ref{algo:umchoice_general}, then the algorithm produces a UM-coloring
from $\L$.
\end{lemma}

\begin{proof}
The proof is similar to the validity proof of Algorithm~\ref{CF-framework} and we omit the details.
  \end{proof}

The key ingredient, which will determine the necessary size of the
lists of $\L$, is the particular choice of the independent set in the
above algorithm. We assume that the hypergraph $H=(V,\E)$
is hereditarily $k$-colorable
for some fixed positive integer $k$. Recall that, as shown before,
this is the case in many geometric hypergraphs.
We must also put some condition on the size of the lists
in the family $\L = \{L_v\}_{v \in V}$.
With some hindsight, we require
\[ \sum_{v \in V} \lambda^{-\cardin{L_v}} < 1, \]
where $\lambda := \frac{k}{k-1}$.

\begin{theorem}\label{thm:um_lists}\cite{cf-choice}
Let $H = (V, \E)$ be a hypergraph which is hereditarily
$k$-colorable and set $\lambda := \frac{k}{k-1}$.
Let $\L = \{L_v\}_{v \in V}$, such that
$\sum_{v \in V} \lambda^{-\cardin{L_v}} < 1$.
Then, $H$ admits a UM-coloring from $\L$.
\end{theorem}

Notice, that in particular for a hypergraph $H$ which is hereditarily $k$-colorable we have:
$$\chum(H) \leq \log_{\lambda} n+1 = O(k \log n)$$
Thus, Theorem~\ref{thm:um_lists} subsumes all the theorems (derived from Algorithm~\ref{CF-framework}) that are mentioned in Section~\ref{sec:geom_hyper}.

\begin{proof}
The proof of Theorem~\ref{thm:um_lists} is constructive and uses a potential method:
This method gives priority to
coloring vertices that have fewer remaining colors in their lists, when choosing the independent sets.
Towards that goal,
we define a potential function on subsets of uncolored vertices
and
we choose the independent set with the highest potential (the
potential quantifies how dangerous it is that some vertex in the
set will run out of colors in its list).
%

For an uncolored vertex $v \in V$,
let $r_t(v)$ denote the number of colors remaining
in the list of $v$ in the beginning of iteration $t$ of the algorithm.
Obviously, the value of $r_t(v)$ depends on the particular run
of the algorithm.
For a subset of uncolored vertices $X \subseteq V$
in the beginning of iteration $t$, let
$P_t(X) := \sum_{v \in X}\lambda^{-r_t(v)}$.
We define the potential in the beginning of iteration $t$ to be
$P_t := P_t(V_t)$,
where $V_t$ denotes the subset of all uncolored vertices
in the beginning of iteration $t$.
Notice that the value of the potential in the beginning of the algorithm
(i.e., in the first iteration) is
$P_1 =\sum_{v \in V} \lambda^{-\cardin{L_v}} < 1$.

Our goal is to show that, with the right choice of
the independent set in each iteration,
we can make sure that for any iteration $t$ and every vertex
$v \in V_t$ the inequality $r_t(v) > 0$
holds.
In order to achieve this, we will show that,
with the right choice of the subset of vertices colored in each iteration,
the potential function $P_t$ is non-increasing in $t$.
This will imply that for any iteration $t$
and every uncolored vertex $v \in V_t$
 we have:
$$\lambda^{-r_t(v)} \leq P_t \leq P_1 < 1 $$
and hence $r_t(v) > 0$, as required.

Assume that the potential function is non-increasing up to iteration $t$.
Let $P_t$ be the value of the potential function in the beginning
of iteration $t$
and let $c$ be the color associated with iteration $t$.
Recall that $V_t$ denotes the set of uncolored vertices that are
considered in iteration $t$,
and $V^c \subseteq V_t$
denotes the subset of uncolored vertices that
contain the color $c$ in their lists.
Put $P' = P_t(V_t\setminus V^c)$
and $P'' = P_t(V^c)$. Note that $P_t= P'+ P''$.
Let us describe
how we find the independent set of vertices to be colored at iteration $t$.
First, we find an auxiliary proper coloring of
the hypergraph $H[V^c]$ with $k$ colors
(here we use the hereditary $k$-colorability property of the
 hypergraph).
Consider the color class $U$ which
has the largest potential $P_t(U)$.
Since the vertices in $V^c$ are partitioned into
at most $k$ independent subsets $U_1, \ldots, U_k$
and $P''=\sum_{i=1}^k P_t(U_i)$,
then by the pigeon-hole principle there is an index $j$ for which
$P_t(U_j) \geq {P''}/{k}$.
We choose $U=U_j$ as the independent set to be colored at iteration $t$.
Notice that, in this case, the value $r_{t+1}(v) = r_{t}(v) - 1$ for
every vertex
$v \in V^c\setminus U$,
and all vertices in $U$ are colored.
For vertices in $V_t \setminus V^c$, there is no change in the size of their lists.
Thus, the value $P_{t+1}$ of the potential function at the end of iteration $t$ (and in the beginning of iteration $t+1$) is
$P_{t+1} \leq P' + \lambda(1-\frac{1}{k})P''$.
Since
$\lambda = \frac{k}{k-1}$,
we have that $P_{t+1} \leq P' + P'' = P_t$, as required.
\end{proof}

\subsubsection{A relation between chromatic and choice number in general hypergraphs}

Using a probabilistic argument, Cheilaris et al.
\cite{cf-choice} proved the following general theorem for
arbitrary hypergraphs and arbitrary colorings with the so-called refinement property:

\begin{definition}
We call $C'$ a \emph{refinement} of a coloring $C$ if
$C(x) \neq C(y)$ implies $C'(x) \neq C'(y)$.
A class $\cal C$ of colorings is said to have \emph{the refinement property}
if every refinement of a coloring in the class is also in the class.
\end{definition}

The class of conflict-free colorings and the class of proper colorings
are examples of classes which have the refinement property. On the
other hand, the class of unique-maximum colorings does not have this
property.

For a class $\calC$ of colorings, one can naturally extend the notions of chromatic
number $\chiC$ and choice number $\chC$ to $\calC$.

\begin{theorem}[\cite{cf-choice}] \label{thm:general}
For every class of colorings $\calC$ that has the refinement property
and every hypergraph $H$ with $n$ vertices,
$\chC(H) \leq \chiC(H) \cdot \ln n +  1$.
\end{theorem}
\begin{proof}
If $k = \chiC(H)$,
then there exists a $\calC$-coloring $C$ of $H$ with colors
$\{1, \dots, k\}$,
which induces a partition of $V$
into $k$ classes: $V_1 \cup V_2 \cup \dots \cup V_k$.
Consider a family $\L = \{ L_{v} \}_{v \in V}$, such that
for every $v$,  $\cardin{L_v} = k^* > k \cdot \ln n$.
We wish to find a family $\L' = \{L'_{v}\}_{v \in V}$ with the
following properties:
\begin{enumerate}
\item For every $v \in V$, $L'_v \subseteq L_v$.
\item For every $v \in V$, $L'_v \neq \emptyset$.
\item For every $i \neq j$, if $v \in V_i$ and $u \in V_j$, then
      $L'_v \cap L'_u = \emptyset$.
\end{enumerate}
Obviously, if such a family $\L'$ exists, then there exists a
$\calC$-coloring from $\L'$: For each $v \in V$, pick a color $x \in
L'_v$ and assign it to $v$.

We create the family $\L'$ randomly as follows: For
each element in $\cup \L$, assign it uniformly at random to one
of the $k$ classes of the partition $V_1 \cup \dots \cup V_k$.
For every vertex $v \in V$, say with $v \in V_i$, we create
$L'_v$, by keeping only elements of $L_v$ that were assigned
through the above random process to $v$'s class, $V_i$.

The family $\L'$ obviously has properties~1 and~3.
We will prove that with positive probability it also has property~2.

For a fixed $v$, the probability that
$L'_v = \emptyset$
is at most
\[
\left(1 - \frac{1}{k}\right)^{k^*} \leq e^{-k^*/k} < e^{-\ln n} =
\frac{1}{n}
\]
and therefore,
using the union bound,
the probability that
for at least one
vertex $v$, $L'_v = \emptyset$,
is at most
\[n \left(1 - \frac{1}{k}\right)^{k^*} <  1.\]
Thus, there is at least one family $\L'$
where property~2 also holds,
as claimed.
\end{proof}

\begin{corollary}
For every hypergraph $H$,
$$\chcf(H) \leq \chicf(H) \cdot \ln n +  1.$$
\end{corollary}

\begin{corollary}
For every hypergraph $H$,
$$\ch(H) \leq \chi(H) \cdot \ln n +  1.$$
\end{corollary}

The argument in the proof of Theorem~\ref{thm:general}
is a generalization of an argument
first given in \cite{ert}, proving that any bipartite graph
with $n$ vertices is $O(\log n)$-choosable (see
also~\cite{Alon1992probchoice}).

\section{Non-Geometric Hypergraphs}

Pach and Tardos \cite{CFPT09} investigated the CF-chromatic number
of arbitrary hypergraphs and proved that the inequality:
$$\CF(H) \leq  1/2 +\sqrt{2m + 1/4}$$
holds for every hypergraph $H$ with $m$ edges, and that this bound
is tight. Cheilaris et al. \cite{cf-choice} strengthened this bound in two ways by proving that:
$$\chum(H) \leq  1/2 +\sqrt{2m + 1/4}$$

If, in addition, every hyperedge contains at least
$2t-1$ vertices (for $t \geq 3$) then Pach and Tardos showed that:
$$\CF(H) = O(m^{\frac{1}{t}} \log m)$$
Using the Lov{\' a}sz Local Lemma, they show that the same result
holds for hypergraphs, in which the size of every edge is at
least $2t- 1$ and every edge intersects at most $m$ other edges.

\paragraph{Hypergraphs induced by neighborhoods in graphs}
A particular interest arises when dealing with hypergraphs
induced by neighborhoods of vertices of a given graph. Given a
graph $G=(V,E)$ and a vertex $v \in V$, denote by $N_G(v)= N(v)$
the set of all neighbors of $v$ in $G$ together with $v$ and
refer to it as the {\em neighborhood of $v$}. Call the set
$\dot{N}(G)=N_G(v) \setminus \{v\}$ the {\em pointed neighborhood
of $v$}. The hypergraph $H$ associated with the neighborhoods of
$G$ has its vertex set $V(H) = V$ and its edge set $E(H)=
\{N_G(v) | v \in V\}$ and the hypergraph $\dot{H}$ associated
with the pointed neighborhoods of $G$ has $V(\dot{H})=V$ and
$E(\dot{H})= \{\dot{N}_G(v)| v \in V\}$. The {\em conflict-free
chromatic parameter $\kappa_{CF}(G)$} is defined simply as
$\CF(H)$ and the {\em pointed} version of this parameter
$\dot{\kappa}_{CF}(G)$ is defined analogously as $\CF(\dot{H})$.

We start with an example taken from \cite{CFPT09} in order to
provide some basic insights into the relation between these two
parameters. Let $K'_s$ be the graph obtained from the complete
graph $K_s$ on $s$ vertices by subdividing each edge with a new
vertex. Each pair of the $s$ original vertices form the pointed
neighborhood of one of the new vertices, so all original vertices
must receive different colors in any conflict-free coloring of the
corresponding hypergraph $\dot{H}$. Thus, we have
$\dot{\kappa}_{CF}(K'_s) \geq  s$ and it is easy to see that
equality holds here. On the other hand, $K'_s$ is bipartite and
any proper coloring of a graph is also a conflict-free coloring of
the hypergraph formed by the neighborhoods of its vertices. This
shows that $\kappa_{CF}(K'_s) = 2$, for any $s \geq 2$. The
example illustrates that the pointed conflict-free chromatic
parameter of a graph cannot be bounded from above by any function
of its non-pointed variant. For many other graphs, the non-pointed
parameter can be larger than the pointed parameter. For instance,
let $G$ denote the graph obtained from the complete graph $K_4$ by
subdividing a single edge with a vertex. It is easy to check that
$\kappa_{CF}(G) = 3$, while $\dot{\kappa}_{CF}(G) = 2$. However,
it is not difficult to verify that
$$
\kappa_{CF}(G) \leq 2\dot{\kappa}_{CF}(G)
$$
for any graph $G$. This inequality holds, because in a
conflict-free coloring of the pointed neighborhoods, each
neighborhood $N(x)$ also has a vertex whose color is not repeated
in $N(x)$, unless $x$ has degree one in the subgraph spanned by
one of the color classes. One can fix this by carefully splitting
each color class into two. The following theorems were proved in
\cite{CFPT09}:

\begin{theorem}[\cite{CFPT09}]
The conflict-free chromatic parameter of any graph $G$ with $n$
vertices satisfies $\kappa_{CF}(G) = O(\log^2 n)$. The
corresponding coloring can be found by a deterministic polynomial
time algorithm.
\end{theorem}

\begin{theorem}[\cite{CFPT09}]
There exist graphs of n vertices with conflict-free chromatic
parameter $\Omega(\log n)$.
\end{theorem}

\begin{problem}
Close the gap between the last two bounds.
\end{problem}

For graphs with maximum degree $\Delta$, a slightly better
upper-bound is known:
\begin{theorem} [\cite{CFPT09}]
The conflict-free chromatic parameter of any graph $G$ with
maximum degree $\Delta$ satisfies $\kappa_{CF}(G) =
O(\log^{2+\epsilon} \Delta)$ for any $\epsilon > 0$. The
corresponding coloring can be found by a deterministic polynomial
time algorithm.
\end{theorem}

\paragraph{Hypergraphs induced by simple paths in graphs}
As mentioned in the introduction, a particular interest is in hypergraphs induced by simple paths in a given graph:
Recall the that given a graph $G$, we consider the hypergraph
$H=(V,E')$ where a subset $V' \subset V$ is a hyperedge in $E'$ if
and only if $V'$ is the set of vertices in some simple path of
$G$. As mentioned before, the parameter $\um(H)$ is known as the vertex ranking number of $G$ and was studied in other context in the literature (see, e.g., \cite{orderedcoloring,Schaffer89}). An interesting question arises when trying to understand the relation between the two parameters $\CF(H)$ and $\um(H)$. This line of research was pursued
in \cite{CKP10} and \cite{ChT2010ciac}. Cheilaris and T{\' o}th proved the following:

\begin{theorem}[\cite{ChT2010ciac}]
(i) Let $G$ be a simple graph and let $H$ be the hypergraph induced by paths in $G$ as above:
Then $\um(H) \leq 2^{\CF(H)} -1$.

(ii) There is is a sequence of such  hypergraphs $\{H_i\}_{i=1}^{\infty}$ induced by paths such that
$$\lim_{n \rightarrow \infty} \frac{\um(H_n)}{\CF(H_n)} = 2.$$
\end{theorem}

Narrowing the gaps between the two parameters for such hypergraphs is an interesting open problem:
\begin{problem}
Let $f(k)$ denote the function of the least integer such that for every hypergraph $H$ induced by path in a graph $G$
we have that $\um(H) \leq f(\CF(H))$. Find the asymptotic behavior of $f$.
\end{problem}

\section{Algorithms}
Until now we were mainly concerned with the combinatorial problem
of obtaining  bounds on the CF-chromatic number of various
hypergraphs. We now turn our attention to the computational
aspect of the corresponding optimization problem. Even et al.
\cite{ELRS} proved that given a finite set $D$ of discs in the
plane, it is NP-hard to compute an optimal CF-coloring for $H(D)$;
namely, a CF-coloring of $H(D)$ using a minimum number of colors.
This hardness result holds even if all discs have the same radius.
However, as mentioned in the introduction, any set $D$ of $n$
discs admits a CF-coloring that uses $O(\log n)$ colors and such
a coloring can be found in deterministic polynomial time (in fact
in $O(n \log n)$ time). This trivially implies that such an
algorithm serves as an $O(\log n)$ approximation algorithm for
the corresponding optimization problem.

\subsection{Approximation Algorithms}
Given a finite set $D$ of discs in the plane, the \emph{size
ratio} of $D$ denoted by $\rho = \rho(D)$ is the ratio between the
maximum and minimum radii of discs in $D$. For simplicity, we may assume that the smallest radius is $1$.
For each $i \geq 1$, let $D^i$ denote the subset of
discs in $D$ whose radius is in the range $[2^{i-1}, 2^i)$.
Let $\phi_{2^i}(D^i)$ denote the maximum
number of centers of discs in $D^i$ that are contained in a $2^i \times 2^i$ square.
Refer to $\phi_{2^i}(D^i)$ as the \emph{local density} of $D^i$ (with respect to $2^i \times 2^i$ square).
For a set of points $X$ in $\Re^2$ let $D_r(X)$ denote the set of
$\cardin{X}$ discs with radius $r$ centered at the points of $X$.
The following algorithmic results were provided in \cite{ELRS}.
\begin{theorem}[\cite{ELRS}] \label{thm:ub}
\begin{enumerate}
\item Given a  finite set $D$
  of discs with size-ratio $\rho$, there exists a polynomial-time algorithm
that compute a CF-coloring of $D$ using $O\left(\min\{(\log
\rho) \cdot \max_{i}\{\log \phi_{2^i}(D^i)\},\log|D|\}\right)$ colors.

\item Given a finite set of centers $X \subset \mathbb{R}^2$,
 there exists a polynomial-time algorithm
that computes a UM-coloring $\chi$ of the hypergraph induced $X$ with respect to all discs using $O(\log|X|)$ colors.
This is equivalent to the following: If we color $D_r(X)$ by assigning each disc $d \in
D_r(X)$ the color of its center then this is a valid UM-coloring
of the hypergraph $H(D_r(X))$ for every radius $r$.
\end{enumerate}
\end{theorem}

The tightness of Theorem~\ref{thm:ub} follows from the fact that
for any integer $n$, there exists a set $D$ of $n$ unit discs with
$\phi_1(D)=n$ for which $\Omega(\log n)$ colors are necessary in
every CF-coloring of $D$.

In the first part of Theorem~\ref{thm:ub} the discs are not
necessarily congruent. That is, the size-ratio $\rho$ may be
bigger than $1$.  In the second part of Theorem~\ref{thm:ub}, the
discs are congruent (i.e., the size-ratio equals $1$). However,
the common radius is not determined in advance. Namely, the order
of quantifiers in the second part of the theorem is as follows:
Given the locations of the disk centers, the algorithm computes a
coloring of the centers (of the discs) such that this coloring is
conflict-free {\em for every\/} radius $r$.

Building on Theorem~\ref{thm:ub}, Even et al. \cite{ELRS} also
obtain two bi-criteria CF-coloring algorithms for discs having
the same (unit) radius. In both cases the algorithm uses only few
colors. In the first case this comes at a cost of not serving a
small area that is covered by the discs (i.e., an area close to
the boundary of the union of the discs). In the second case, all
the area covered by the discs is served, but the discs are assumed
to have a slightly larger radius. A formal statement of these
bi-criteria results is as follows:

\begin{theorem}[\cite{ELRS}]\label{thm:bi}
  For every $0<\eps<1$ and every finite set of centers $X \subset
  \mathbb{R}^2$, there exist polynomial-time algorithms that compute
  colorings as follows:
  \begin{enumerate}
  \item A coloring $\chi$ of $D_1(X)$ using $O\left(\log
      \frac{1}{\eps}\right)$ colors for which the following holds: The
    area of the set of points in $\bigcup D_1(X)$ that are not
    served with respect to $\chi$ is at most an $\eps$-fraction of the
    total area of $D_1(X)$.
  \item A coloring of $D_{1+\eps}(X)$ that uses $O\left(\log
      \frac{1}{\eps}\right)$ colors such that every point in
    $\bigcup \D_1(X)$ is served.
  \end{enumerate}
\end{theorem}
In other words, in the first case, the portion of the total area
that is not served is an exponentially small fraction as a
function of the number of colors. In the second case, the
increase in the radius of the discs is exponentially small as a
function of the number of colors.

The following problem seems like a non-trivial challenge.
\begin{problem}
Is there a constant factor approximation algorithm for finding an
optimal CF-coloring for a finite set of discs in the plane?
\end{problem}

{\noindent \bf Remark:} In the special case that all discs are
congruent (i.e., have the same radius) Lev-Tov and Peleg
\cite{Lev-TovP09} have recently provided a constant-factor
approximation algorithm.

\subsubsection{An $O(1)$-Approximation
for CF-Coloring of Rectangles and Regular Hexagons}
Recall that Theorem~\ref{cf-rectangles} states that every set of
$n$ axis-parallel rectangles can be CF-colored with $O(\log^2 n)$
colors and such a coloring can be found in polynomial time.

Let $\Rcal$ denote a set of axis-parallel rectangles. Given a
rectangle $R\in \Rcal$, let $w(R)$ ($h(R)$, respectively) denote
the width (height, respectively) of $R$. The \emph{size-ratio} of
$\Rcal$ is defined by $\max\left\{\frac{w(R_1)}{w(R_2)} ,
\frac{h(R_1)}{h(R_2)} \right\}_{R_1,R_2\in \Rcal}$.

The size ratio of a collection of regular hexagons is simply the
ratio of the longest side length and the shortest side length.

\begin{theorem} [\cite{ELRS}]\label{thm:squares} \label{thm:hexagons}
  Let $\Rcal$ denote either a set of axis-parallel rectangles or a set
  of homothets of a regular hexagons.  Let $\rho$ denote the size-ratio
  of $\Rcal$ and let $\chiopt(\Rcal)$ denote an optimal CF-coloring
  of $\Rcal$.
  \begin{enumerate}
  \item If $\Rcal$ is a set of rectangles, then there exists a
    polynomial-time algorithm that computes a CF-coloring $\chi$ of $\Rcal$
    such that $|\chi(\Rcal)| =O((\log \rho+1)^2 \cdot |\chiopt(\Rcal)|)$.
  \item If $\Rcal$ is a set of hexagons, then there exists a polynomial-time
    algorithm that computes a CF-coloring $\chi$ of $\Rcal$ such that
    $|\chi(\Rcal)| =O((\log \rho+1) \cdot |\chiopt(\Rcal)|)$.
  \end{enumerate}
\end{theorem}
For a constant size-ratio $\rho$, Theorem~\ref{thm:squares}
implies a constant approximation algorithm.


\subsection{Online CF-Coloring}
Recall the motivation to study CF-coloring in the context of cellular antanae.
To capture a dynamic scenario where antennae can be added to the
network, Chen et al. \cite{cf7} introduced an online version of
the CF coloring problem. As we shall soon see, the online version
of the problem is considerably harder, even in the
one-dimensional case, where the static version (i.e., CF-coloring the discrete intervals hypergraph) is trivial and
fully understood.

\subsubsection{Points with respect to intervals}
Let us start with the simplest possible example where things
become highly non-trivial in an online setting. We start with the dynamic extension of the discrete interval hypergraph case. Thats is,
we deal with coloring of points on the line, with respect to interval ranges. We maintain
a finite set $P\subset\reals$. Initially, $P$ is empty, and an adversary
repeatedly insert points into $P$, one point at a time. We denote
by $P(t)$ the set $P$ after the $t$th point has been inserted.
Each time a new point $p$ is inserted, we need to assign a color $c(p)$
to it, which is a positive integer. Once the color has been
assigned to $p$, it cannot be changed in the future. The coloring
should remain a valid CF-coloring at all times. That is, as in the static case,
for any interval $I$ that contains points of $P(t)$, there is a
color that appears exactly once in~$I$.

We begin by examining a natural, simple, and obvious coloring
algorithm (referred to as the UniMax greedy algorithm) which
might be inefficient in the worst case. Chen et al. \cite{cf7}
presented an efficient 2-stage
variant of the UniMax greedy algorithm and showed that the maximum
number of colors that it uses is $\Theta(\log^2n)$.

As in the case in most CF-coloring of hypergraphs that were
tackled so far, we wish to maintain the unique maximum
invariant. At any given step $t$ the coloring of $P(t)$ is a UM-coloring.

The following simple-minded algorithm for coloring an inserted
point $p$ into the current set $P(t)$ is used. We say that the newly inserted point $p$ {\em sees\/} a
point $x$ if all the colors of the points between $p$ and $x$
(exclusive) are smaller than $c(x)$. In this case we also say
that $p$ sees the color $c(x)$. Then $p$ gets the smallest color
that it does not see. (Note that a color can be seen from $p$
either to the left or to the right, but not in both directions;
see below.) Refer to this algorithm as the {\em Unique Maximum
Greedy\/} algorithm, or the UniMax greedy algorithm, for short.

Below is an illustration of the coloring rule of the UniMax
greedy algorithm. The left column gives the colors (integers in
the range $1,2, \ldots,6$) assigned to the points in the current
set $P$ and the location of the next point to be inserted
(indicated by a period). The right column gives the colors
``seen'' by the new point. The colors seen to the left precede
the $\cdot$, and those seen to the right succeed the~$\cdot$.
\[
\begin{array}{ll}
1 \cdot &
  [1 \cdot ] \\
1 \cdot 2 &
  [1 \cdot 2 ] \\
1 \cdot 3 2 &
  [1 \cdot 3 ] \\
1 2 \cdot 3 2 &
  [2 \cdot 3 ] \\
1 2 1 \cdot 3 2 &
  [2 1 \cdot 3 ] \\
1 2 1 \cdot 4 3 2 &
  [2 1 \cdot 4 ] \\
1 2 1 \cdot 3 4 3 2 &
  [2 1 \cdot 3 4 ] \\
1 2 1 5 \cdot 3 4 3 2 &
  [5 \cdot 3 4 ] \\
1 2 1 5 \cdot 1 3 4 3 2 &
  [5 \cdot 1 3 4 ] \\
1 2 1 5 2 \cdot 1 3 4 3 2 &
  [5 2 \cdot 1 3 4 ] \\
1 2 1 5 2 6 \cdot 1 3 4 3 2 &
  [6 \cdot 1 3 4 ]
\end{array}
\]

{\it Correctness}. The correctness of the algorithm is
established by induction on the insertion order. First, note that
no color can be seen twice from $p$: This is obvious for two
points that lie both to the left or both to the right of $p$. If
$p$ sees the same color at a point $u$ to its left and at a point
$v$ to its right, then the interval $[u,v]$, before $p$ is
inserted, does not have a unique maximum color; thus this case is
impossible, too. Next, if $p$ is assigned color $c$, any interval
that contains $p$ still has a unique maximum color: This follows
by induction when the maximum color is greater than $c$. If the
maximum color is $c$, then it cannot be shared by another point
$u$ in the interval, because then $p$ would have seen the nearest
such point and thus would not be assigned color $c$. It is also
easy to see that the algorithm assigns to each newly inserted
point the smallest possible color that maintains the invariant of
a unique maximum color in each interval. This makes the algorithm
{\em greedy\/} with respect to the unique maximum condition.

{\it Special insertion orders}. Denote by $C(P(t))$ the sequence
of colors assigned to the points of $P(t)$, in left-to-right
order along the line.

The {\em complete binary tree sequence} $S_k$ of order $k$ is
defined recursively as $S_1=(1)$ and $S_k = S_{k-1} \| (k) \|
S_{k-1}$, for $k>1$, where $\|$ denotes concatenation.  Clearly,
$|S_k|=2^k-1$.

For each pair of integers $a<b$, denote by $C_0(a,b)$ the
following special sequence. Let $k$ be the integer satisfying
$2^{k-1}\le b<2^k$. Then $C_0(a,b)$ is the subsequence of $S_k$
from the $a$th place to the $b$th place (inclusive). For example,
$C_0(5,12)$ is the subsequence $(1,2,1,4,1,2,1,3)$ of
$(1,2,1,3,1,2,1,4,1,2,1,3,1,2,1)$.

\begin{lemma}\label{right}
{\rm(a)} If each point is inserted into $P$ to the right of all
preceding points, then $C(P(t)) = C_0(1,t)$.

{\rm(b)} If each point is inserted into $P$ to the left of all
preceding points, then $C(P(t)) \allowbreak = C_0(2^k-t${\rm,
}$2^k-1)$, where $k$ satisfies\/ $2^{k-1}\le t<2^k$.

\end{lemma}

\begin{proof}
The proof is easy and is left as an exercise to the
reader.\qquad\end{proof}

Unfortunately, the UniMax greedy algorithm might be very
inefficient as was shown in \cite{cf7}:
\begin{theorem} [\cite{cf7}] \label{lb:sqrt}
The UniMax greedy algorithm may require\/ $\Omega(\sqrt{n})$
colors in the worst case for a set of $n$ points.
\end{theorem}

\begin{problem}
Obtain an upper bound for the maximum number of colors that the
algorithm uses for $n$ inserted points. It is conjectured that the
bound is close to the $\Omega(\sqrt{n})$ lower bound. At the
moment, there is no known sub-linear upper bound.
\end{problem}
\paragraph{Related algorithms}\mbox{}

{\it The First-Fit algorithm---another greedy strategy}. The
UniMax greedy algorithm is greedy for maintaining the unique
maximum invariant. Perhaps it is more natural to
consider a greedy approach in which we want only to enforce the
standard CF property. That is, we want to assign to each newly
inserted point the {\em smallest\/} color for which the CF
property continues to hold. There are cases where this {\em
First-Fit\/} greedy algorithm uses fewer colors than the UniMax
greedy algorithm: Consider an insertion of five points in the
order $(1\;3\;2\;4\;5)$. The UniMax greedy algorithm produces the
color sequence $(1\;3\;2\;1\;4)$, whereas the First-Fit algorithm
produces the coloring $(1\;3\;2\;1\;2)$. Unfortunately, Bar-Noy et
al.~\cite{NCS06} have shown that there are
sequences with $2i+3$ elements that force the algorithm to use
$i+3$ colors, and this bound is tight.

{\it CF coloring for unit intervals}. Consider the special case
where we want the CF property to hold only for {\em unit
intervals}. In this case, $O(\log n)$ colors suffice: Partition
the line into the unit intervals $J_i = [i,i+1)$ for $i\in\ZZ$.
Color the intervals $J_i$ with even $i$ as white, and those with
odd $i$ as black. Note that any unit interval meets only one
white and one black interval. We color the points in each $J_i$
independently, using the same set of ``light colors'' for each
white interval and the same set of ``dark colors'' for each black
interval. For each $J_i$, we color the points that it contains
using the UniMax greedy algorithm, except that new points
inserted into $J_i$ between two previously inserted points get a
special color, color~0.  It is easily checked that the resulting
coloring is CF with respect to unit intervals. Since we
effectively insert points into any $J_i$ only to the left or to
the right of the previously inserted points, Lemma \ref{right}(c)
implies that the algorithm uses only $O(\log n)$ (light and dark)
colors. We remark that this algorithm satisfies the unique
maximum color property for unit-length intervals.

We note that, in contrast to the static case (which can always be
solved with $O(1)$ colors), $\Omega(\log n)$ colors may be needed
in the worst case. Indeed, consider a left-to-right insertion of
$n$ points into a sufficiently small interval. Each contiguous
subsequence $\sigma$ of the points will be a suffix of the whole
sequence at the time the rightmost element of $\sigma$ is
inserted. Since such a suffix can be cut off the current set by a
unit interval, it must have a unique color. Hence, at the end of
insertion, {\em every\/} subsequence must have a unique color,
which implies (see \cite{ELRS,SmPHD}) that $\Omega(\log n)$ colors
are needed.

\paragraph{An efficient online deterministic algorithm for points with respect to intervals}

We describe an efficient online algorithm for coloring points
with respect to intervals that was obtained in \cite{cf7}. This is
done by modifying the UniMax greedy algorithm into a deterministic
2-stage coloring scheme. It is then shown that it uses only
$O(\log^2n)$ colors. The algorithm is referred to as the {\em
leveled UniMax greedy algorithm}.

Let $x$ be the point which we currently insert.  We assign a
color to $x$ in two steps.  First we assign $x$ to a {\em level},
denoted by $\ell(x)$.  Once $x$ is assigned to level $\ell(x)$ we
give it an actual color among the set of colors dedicated to
$\ell(x)$. We maintain the invariant that each color is used by
at most one level. Formally, the colors that we use are pairs
$(\ell(x),c(x))\in\ZZ^2$, where $\ell(x)$ is the level of $x$ and
$c(x)$ is its integer color within that level.

Modifying the definition from the UniMax greedy algorithm, we say
that point $x$ {\em sees\/} point $y$ (or that point $y$ is {\em
visible\/} to $x$) if and only if for every point $z$ between $x$
and $y$, $\ell(z) < \ell(y)$.  When $x$ is inserted, we set
$\ell(x)$ to be the smallest level $\ell$ such that either to the
left of $x$ or to the right of $x$ (or in both directions) there
is no point $y$ visible to $x$ at level~$\ell$.

To give $x$ a color, we now consider only the points of level
$\ell(x)$ that $x$ can see. That is, we discard every point $y$
such that $\ell(y) \not= \ell(x)$, and every point $y$ such that
$\ell(y) = \ell(x)$ and there is a point $z$ between $x$ and $y$
such that $\ell(z) > \ell(y)$.  We apply the UniMax greedy
algorithm so as to color $x$ with respect to the sequence $P_x$
of the remaining points, using the colors of level $\ell(x)$
only.  That is, we give $x$ the color $(\ell(x),c(x))$, where
$c(x)$ is the smallest color that ensures that the coloring of
$P_x$ maintains the unique maximum color condition.  This
completes the description of the algorithm. See Figure~\ref{2stalg}
for an illustration.

\begin{figure}[t!]
\centering
\input{2stalg.pstex_t}
\caption{Illustrating the\/ {\rm2}-stage deterministic algorithm.
An insertion order that realizes the depicted assignment of
levels to points is to first insert all level-\/{\rm1} points
from left to right, then insert the level-\/{\rm2} points from
left to right, and then the level-\/{\rm3} points.}\label{2stalg}
\end{figure}
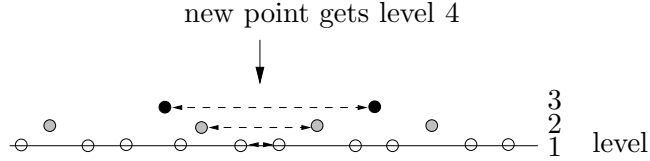

We begin the analysis of the algorithm by making a few
observations on its performance.

(a) Suppose that a point $x$ is inserted and is assigned to level
$i>1$. Since $x$ was not assigned to any level $j<i$, it must see
a point $\ell_j$ at level $j$ that lies to its left, and another
such point $r_j$ that lies to its right. Let $E_j(x)$ denote the
interval $[\ell_j,r_j]$. Note that, by definition, these
intervals are {\em nested}, that is, $E_j(x)\subset E_k(x)$ for
$j<k<i$. See Figure~\ref{2stalg}.

(b) We define a {\em run\/} at level $i$ to be a maximal sequence
of points $x_1<x_2<\cdots<x_k$ at level $i$, such that all points
between $x_1$ and $x_k$ that are distinct from
$x_2,x_3,\ldots,x_{k-1}$ are assigned to levels smaller than $i$.
Whenever a new point $x$ is assigned to level $i$ and is inserted
into a run of that level, it is always inserted either to the left
or to the right of all points in the run. Moreover, the actual
color that $x$ gets is determined solely from the colors of the
points already in the run. See Figure~\ref{2stalg}.

(c) The runs keep evolving as new points are inserted. A run may
either grow when a new point of the same level is inserted at its
left or right end (note that other points at smaller levels may
separate the new point from the former end of the run) or split
into two runs when a point of a higher level is inserted
somewhere between its ends.

(d) As in observation~(a), the points at level $i$ define {\em
intervals}, called {\em $i$-intervals}. Any such interval $E$ is
a contiguous subsequence $[x,y]$ of $P$, so that $x$ and $y$ are
both at level $i$ and all the points between $x$ and $y$ have
smaller levels. $E$ is formed when the second of its endpoints,
say $x$, is inserted. We say that $x$ {\em closes\/} the interval
$E$ and refer to it as a {\em closing point}. Note that, by
construction, $x$ cannot close another
interval.\enlargethispage*{2pt}

(e) Continuing observation~(a), when $x$ is inserted, it {\em
destroys\/} the intervals $E_j(x)$, for $j<i$, into which it is
inserted, and only these intervals. That is, each of these
intervals now contains a point with a level greater than that of
its endpoints, so it is no longer a valid interval. We charge $x$
to the set of the closing endpoints of all these intervals.
Clearly, none of these points will ever be charged again by
another insertion (since it is the closing endpoint of only one
interval, which is now destroyed). We maintain a forest $F$,
whose nodes are all the points of $P$. The leaves of $F$ are all
the points at level~1. When a new point $x$ is inserted, we make
it a new root of $F$, and the parent of all the closing points
that it charges. Since these points have smaller levels than $x$,
and since none of these points becomes a child of another parent,
it follows that $F$ is indeed a forest.

Note that the nonclosing points can only be roots of trees of $F$.
Note also that a node at level $i$ has exactly $i-1$ children,
exactly one at each level $j<i$. Hence, each tree of $F$ is a
{\em binomial tree\/} (see~\cite{CLR}); if its root has level $i$,
then it has $2^i$~nodes.

This implies that if $m$ is the maximal level assigned after $n$
points have been inserted, then we must have $2^m\le n$, or
$m\le\log n$. That is, the algorithm uses at most $\log n$ levels.

We next prove that the algorithm uses only $O(\log n)$ colors at
each level. We recall the way runs evolve: They grow by adding
points at their right or left ends, and split into prefix and
suffix subruns, when a point with a larger level is inserted in
their middle.

\begin{lemma}\label{tree-seq}
At any time during the insertion process, the colors assigned to
the points in a run form a sequence of the form $C_0(a,b)$.
Moreover, when the $j$th smallest color of level $i$ is given to
a point $x$, the run to which $x$ is appended has at least\/
$2^{j-2}+1$ elements (including~$x$).
\end{lemma}

\begin{proof}
 The proof proceeds by induction through the sequence of insertion steps
 and is based on the following observation.
 Let $\sigma$ be a contiguous subsequence
 of the complete binary tree sequence $S_{k-1}$, and let $x$ be a point
 added, say, to the left of $\sigma$.
 If we assign to $x$ color $c(x)$, using the UniMax greedy algorithm, then
 $(c(x))\| \sigma$ is a contiguous subsequence of either $S_{k-1}$ or
 $S_{k}$. The latter happens only if $\sigma$ contains $S_{k-2} \| (k-1)$
 as a prefix. Symmetric properties hold when $x$ is inserted to the right
 of $\sigma$. We omit the straightforward proof of this observation.\qquad\end{proof}

As a consequence we have.

\begin{theorem}[\cite{cf7}]\label{good}
{\rm(a)} The algorithm uses at most\/ $(2+\log n)\log n$ colors.

{\rm(b)} At any time, the coloring is a valid CF-coloring\@.

{\rm(c)} In the worst case the algorithm may be forced to use\/
$\Omega(\log^2n)$ colors after $n$ points are inserted.
\end{theorem}

\begin{proof}
(a) We have already argued that the number of levels is at most
$\log n$. Within a level $i$, the $k$th smallest color is
assigned when a run contains at least $2^{k-2}$ points. Hence
$2^{k-2}\le n$, or $k\le 2+\log n$, and (a)~follows.

To show~(b), consider an arbitrary interval $I$. Let $\ell$ be
the highest level of a point in $I$. Let $\sigma=(y_1 , y_2
,\ldots , y_j)$ be the sequence of the points in $I$ of level
$\ell$. Since $\ell$ is the highest level in $I$, $\sigma$ is a
contiguous subsequence of some run, and, by Lemma~\ref{tree-seq},
the sequence of the colors of its points is also of the form
$C_0(a',b')$. Hence, there is a point $y_i \in \sigma$ which is
uniquely colored among $y_1, y_2, \ldots ,  y_j$ by a color of
level~$\ell$.\enlargethispage*{4pt}

To show~(c), we construct a sequence $P$ so as to force its
coloring to proceed level by level. We first insert
$2^{k-1}$~points from left to right, thereby making them all be
assigned to level~1 and colored with $k$ different colors of that
level. Let $P_1$ denote the set of these points. We next insert a
second batch of $2^{k-2}$~points from left to right. The first
point is inserted between the first and second points of $P_1$,
the second point between the third and fourth points of $P_1$,
and so on, where the $j$th new point is inserted between the
$(2j-1)$th and $(2j)$th points of $P_1$. By construction, all
points in the second batch are assigned to level~2, and they are
colored with $k-1$ different colors of that level. Let $P_2$
denote the set of all points inserted so far. $P_2$ is the
concatenation of $2^{k-2}$~triples, where the levels in each
triple are $(1,2,1)$. We now insert a third batch of
$2^{k-3}$~points from left to right. The first point is inserted
between the first and second triples of $P_2$, the second point
between the third and fourth triples of $P_2$, and so on, where
the $j$th new point is inserted between the $(2j-1)$th and
$(2j)$th triples of $P_2$. By construction, all points in the
third batch are assigned to level~3, and they are colored with
$k-2$ different colors of that level.

The construction is continued in this manner. Just before
inserting the $i$th batch of $2^{k-i}$ points, we have a set
$P_{i-1}$ of $2^{k-1}+\cdots + 2^{k-i+1}$ points, which is the
concatenation of $2^{k-i+1}$ tuples, where the sequences of
levels in each of these tuples are all identical and equal to the
``complete binary tree sequence'' $C_0(1,2^{i-1}-1)$, as defined
above (whose elements now encode levels rather than colors). The
points of the $i$th batch are inserted from left to right, where
the $j$th point is inserted between the $(2j-1)$th and $(2j)$th
tuples of $P_{i-1}$. By construction, all points in the $i$th
batch are assigned to level $i$ and are colored with $k-i+1$
different colors of that level. Proceeding in this manner, we end
the construction by inserting the $(k-1)$th batch, which consists
of a single point that is assigned to level $k$. Altogether we
have inserted $n=2^k-1$ points and forced the algorithm to use
$k+(k-1)+\cdots +1 = k(k+1)/2 = \Omega(\log^2n)$ different
colors.\qquad\end{proof}

%

 Given that the only known lower bound
for this online CF-coloring problem is $\Omega(\log n)$ which
holds also in the static problem, its a major open problem to
close the gap with the $O(\log^2 n)$ upper bound provided by the
algorithm above.

\begin{problem}
Find a deterministic online CF-coloring for coloring points with
respect to intervals which uses $o(\log^2 n)$ colors in the worst
case or improve the $\Omega(\log n)$ lower bound.
\end{problem}

\paragraph{Other Online models}
For the case of online CF-coloring points with respect to intervals, other models of a weaker adversary were studied in \cite{BarNoy2}. For example, a natural assumption is that the adversary reveals, for a newly inserted point, its final position among the set of all points in the end of the online input. This is referred to as the {\em online absolute positions model}. In this model an online CF-coloring algorithm that uses at most $O(\log n)$ colors is presented in \cite{BarNoy2}.

\subsubsection{points with respect to halfplanes or unit discs}
In \cite{cf7} it was shown that the two-dimensional variant of
online CF-coloring a given sequence of inserted points with
respect to arbitrary discs is hopeless as there exists sequences
of $n$ points for which every CF-coloring requires $n$ distinct
colors. However if we require a CF-coloring with respect to
congruent discs or with respect to half-planes, there is some
hope. Even though no efficient deterministic online algorithms
are known for such cases, some efficient randomized algorithms
that uses expected $O(\log n)$ colors are provided in \cite{cf6,BarNoy3}
under the assumption that the adversary is oblivious to the
random bits used by the algorithm.

Chen Kaplan and Sharir \cite{cf6} introduced an $O(log^3 n)$
deterministic algorithm for online CF-coloring any $n$
nearly-equal axis-parallel rectangles in the plane.

\subsubsection{Degenerate hypergraphs}
Next, we describe the general framework of \cite{BarNoy3} for
online CF-coloring any hypergraph. This framework is used to
obtain efficient randomized online algorithms for hypergraphs
provided that a special parameter referred to as the {\em
degeneracy} of the underlying hypergraph is small. This notion
extends the notion of a degenerate graph to that of a hypergraph:

\begin{definition}\label{defn:hypergraphdegenerate}
Let $k > 0$ be a fixed integer and let $H=(V,E)$ be a hypergraph
on the $n$ vertices $v_1$, \ldots, $v_n$. For a permutation $\pi
\colon \oneton \to \oneton$ define the $n$ partial sums, indexed
by $t=1,\dots,n$,
\[
S^\pi_t = \sum_{j=1}^{t} d(v_{\pi(j)}),
\]
where
\[
d(v_{\pi(j)}) =
 \bigl|\bigl\{ i < j \mid
   \{v_{\pi(i)},v_{\pi(j)}\} \in
      G(H(\{v_{\pi(1)},...,v_{\pi(j)}\}))
 \bigr\}\bigr|,
\]
that is, $d(v_{\pi(j)})$ is the number of neighbors of
$v_{\pi(j)}$ in the Delaunay graph of the hypergraph induced by
$\{v_{\pi(1)},...,v_{\pi(j)}\}$. Assume that for all permutations
$\pi$ and for every $t \in \oneton$ we have
\begin{equation}
S^\pi_t \leq k t.  \label{eq:kdegeneracy}
\end{equation}
Then, we say that $H$ is $k$-\emph{degenerate}.
\end{definition}

Let $H=(V,E)$ be any hypergraph. We define a framework that colors
the vertices of $V$ in an online fashion, i.e., when the vertices
of $V$ are revealed by an adversary one at a time. At each time
step $t$, the algorithm must assign a color to the newly revealed
vertex $v_t$. This color cannot be changed in future times $t' >
t$. The coloring has to be conflict-free for all the induced
hypergraphs $H(V_t)$ with $t=1,\ldots,n$, where $V_t \subseteq V$
is the set of vertices revealed by time $t$.

For a fixed positive integer $h$, let $A=\{a_1,\dots,a_{h}\}$ be a
set of $h$ \emph{auxiliary} colors. This auxiliary colors set
should not be confused with the set of \emph{main} colors used
for the conflict-free coloring: \{1, 2, \dots\}.
Let $f \colon \Positives \rightarrow A$ be some fixed function.
In the following, we define the framework that depends on the
choice of the function $f$ and the parameter $h$.

A table (to be updated online) is maintained with row entries
indexed by the variable $i$ with range in $\Positives$. Each row
entry $i$ at time $t$ is associated with a subset $V^i_t
\subseteq V_t$ in addition to an auxiliary proper
non-monochromatic coloring of $H(V^i_t)$ with at most $h$ colors. The subsets $V^i_t$ are nested. Namely, $V^{i+1}_t \subset V^i_t$ for every $i$. Informally, we think of a newly inserted vertex as trying to reach its final entry
by some decision process. It starts with entry $1$ and continue ``climbing" to higher levels as long as it does not
succeed to get its final color.
We say that $f(i)$ is the auxiliary color that {\em represents}
entry $i$ in the table. At the beginning all entries of the table
are empty. Suppose all entries of the table are updated until
time $t-1$ and let $v_t$ be the vertex revealed by the adversary
at time $t$.
The framework first checks if an auxiliary color can be assigned
to $v_t$ such that the auxiliary coloring of $V^1_{t-1}$ together
with the color of $v_t$ is a proper non-monochromatic coloring of
$H(V^1_{t-1} \cup \{v_t\})$. Any (proper non-monochromatic)
coloring procedure can be used by the framework. For example a
first-fit greedy method in which all colors in the order $a_1$,
\ldots, $a_h$ are checked until one is found.
If such a color cannot be found for $v_t$, then entry $1$ is left
with no changes and the process continues to the next entry.
If however, such a color can be assigned, then $v_t$ is added to
the set $V^1_{t-1}$. Let $c$ denote such an auxiliary color
assigned to $v_t$. If this color is the same as $f(1)$ (the
auxiliary color that represents entry $1$), then the final color
in the online conflict-free coloring of $v_t$ is $1$ and the
updating process for the $t$-th vertex stops. Otherwise, if an
auxiliary color cannot be found or if the assigned auxiliary
color is not the same as $f(1)$, then the updating process
continues to the next entry. The updating process stops at the
first entry $i$ for which $v_t$ is both added to $V^i_t$ and the
auxiliary color assigned to $v_t$ is the same as $f(i)$.
Then, the main color of $v_t$ in the final conflict-free coloring
is set to $i$. See Figure~\ref{fig:table} for an illustration.

 \begin{figure}[htbp]
   \begin{center}
      \includegraphics[width=0.5\textwidth ]{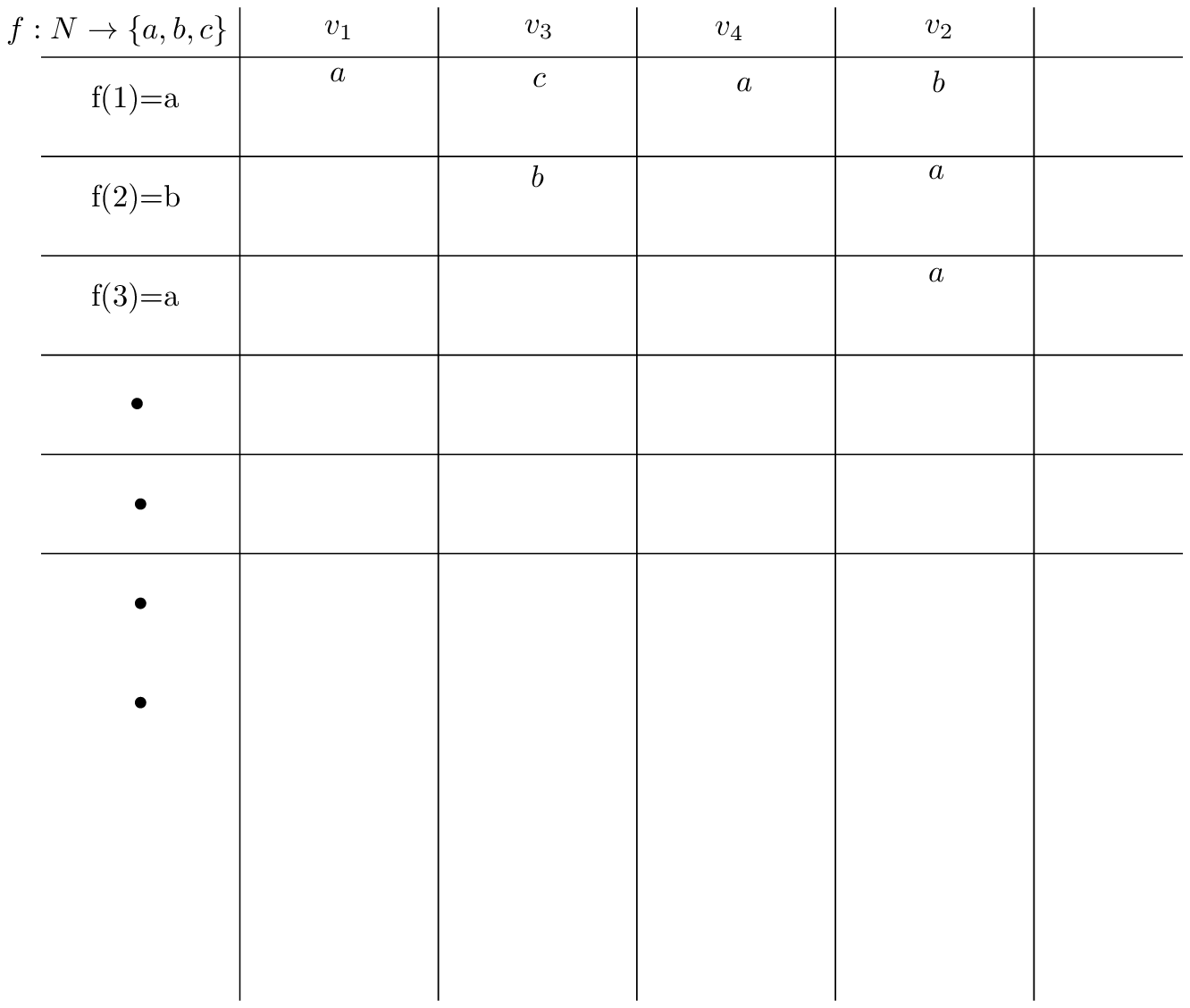}
      \caption{An example of the updating process of the table for the hypergraph induced by points with respect to intervals. $3$ auxiliary colors denoted $\{a,b,c\}$ are used. In each line $i$ the auxiliary coloring is given. It serves as a proper coloring for the hypergraphs $H(V_t^i)$ induced by the subset $V_t^i$ of all points revealed up to time $t$ that reached line $i$. The first point $v_1$ is inserted to the left. The second point $v_2$ to the right and the third point $v_3$ in the middle, etc. For instance, at the first entry (i.e., line) of the table, the auxiliary color of $v_2$ is $b$. In the second line it is $a$ and in the third line it is $a$. Since $f(3)=a$, the final color of $v_2$ is $3$. Similarly, the final color of $v_1$ is $1$, of $v_3$ is $2$, and of $v_4$ is $1$.}
            \label{fig:table}
   \end{center}
  \end{figure}

It is possible that $v_t$ never gets a final color. In this case
we say that the framework does not halt.
However, termination can be guaranteed by imposing some
restrictions on the auxiliary coloring method and the choice of
the function $f$. For example, if first-fit is used for the
auxiliary colorings at any entry and if $f$ is the constant
function $f(i)=a_1$, for all $i$, then the framework is
guaranteed to halt for any time $t$.
Later, a randomized online algorithm based on this framework is
derived under the oblivious adversary model. This algorithm always
halts, or to be more precise halts with probability 1, and
moreover it halts after a ``small'' number of entries with high
probability. We prove that the above framework produces a valid
conflict-free coloring in case it halts.

\begin{lemma}\label{lemma:framework_correctness}
If the above framework halts for any vertex $v_t$ then it produces
a valid online conflict-free coloring of $H$.
\end{lemma}
\begin{proof}
Let $H(V_t)$ be the hypergraph induced by the vertices already
revealed at time $t$. Let $S$ be a hyperedge in this hypergraph
and let $j$ be the maximum integer for which there is a vertex
$v$ of $S$ colored with $j$. We claim that exactly one such
vertex in $S$ exists. Assume to the contrary that there is
another vertex $v'$ in $S$ colored with $j$. This means that at
time $t$ both vertices $v$ and $v'$ were present at entry $j$ of
the table (i.e., $v,v' \in V^j_t$) and that they both got an
auxiliary color (in the auxiliary coloring of the set $V^j_t$)
which equals $f(j)$. However, since the auxiliary coloring is a
proper non-monochromatic coloring of the induced hypergraph at
entry $j$, $S\cap V^j_t$ is not monochromatic so there must exist
a third vertex $v'' \in S \cap V^j_t$ that was present at entry
$j$ and was assigned an auxiliary color different from $f(j)$.
Thus, $v''$ got its final color in an entry greater than $j$, a
contradiction to the maximality of $j$ in the hyperedge $S$. This
completes the proof of the lemma.\qquad
\end{proof}

The above algorithmic framework can also describe some well-known
deterministic algorithms. For example, if first-fit is used for
auxiliary colorings and $f$ is the constant function, $f(i)=a_1$,
for all $i$, then, for the hypergraph induced by points on a line
with respect to intervals, the algorithm derived from the
framework becomes identical to the UniMax greedy algorithm
described above.

\paragraph{An online randomized conflict-free coloring algorithm}
We devise a randomized online conflict-free coloring algorithm in
the oblivious adversary model. In this model, the adversary has
to commit to a permutation according to the order of which the
vertices of the hypergraph are revealed to the algorithm. Namely,
the adversary does not have access to the random bits that are
used by the algorithm. The algorithm always produces a valid
coloring and the number of colors used is related to the
degeneracy of the underlying hypergraph in a manner described in
the following theorem.

\begin{theorem}[\cite{BarNoy3}]
\label{degenerate_main} Let $H=(V,E)$ be a $k$-degenerate
hypergraph on $n$ vertices. Then, there exists a randomized online
conflict-free coloring algorithm for $H$ which uses at most
$O(\log_{1+\frac{1}{4k+1}}n) = O(k \log n)$ colors with high
probability against an oblivious adversary.
\end{theorem}

The algorithm is based on the  framework presented above. In order
to define the algorithm, we need to state what is (a)~the set of
auxiliary colors of each entry, (b)~the function $f$, and (c)~the
algorithm we use for the auxiliary coloring at each entry.
We use the set of auxiliary colors $A = \{a_1, \dots,
a_{2k+1}\}$. For each entry $i$, the representing color $f(i)$ is
chosen uniformly at random from $A$. We use a first-fit algorithm
for the auxiliary coloring.

Our assumption on the hypergraph $H$ (being $k$-degenerate)
implies that at least half of the vertices up to time $t$ that
\emph{reached} entry $i$ (but not necessarily added to entry $i$),
denoted by $X^t_{\geq i}$, have been actually given some auxiliary
color at entry $i$ (that is, $\cardin{V^i_t} \geq
\frac{1}{2}\cardin{X^t_{\geq i}}$). This is due to the fact that
at least half of those vertices $v_t$ have at most $2k$ neighbors
in the Delaunay graph of the hypergraph induced by $X^{t-1}_{\geq
i}$ (since the sum of these quantities is at most
$k\cardin{X^t_{\geq i}}$ and since $V^i_t \subseteq X^t_{\geq
i}$). Therefore, since we have $2k+1$ colors available, there is
always an available  color to assign to such a vertex. The
following lemma shows that if we use one of these available
colors then the updated coloring is indeed a proper
non-monochromatic coloring of the corresponding induced
hypergraph as well.

\begin{lemma}
\label{auxiliary} Let $H=(V,E)$ be a $k$-degenerate hypergraph
and let $V^j_t$ be the subset of $V$ at time $t$ and at level $j$
as produced by the above algorithm. Then, for any $j$ and $t$ if
$v_t$ is assigned a color distinct from all its neighbors in the
Delaunay graph $G(H(V^j_t))$ then this color together with the
colors assigned to the vertices $V^j_{t-1}$ is also a proper
non-monochromatic coloring of the hypergraph $H(V^j_t)$.
\end{lemma}
\begin{proof}
Follows from Lemma~\ref{lemma:framework_correctness}
\end{proof}

We also prove that for every vertex $v_t$, the algorithm always
halts, or more precisely halts with probability 1.

\begin{proposition}
For every vertex $v_t$, the algorithm halts with probability 1.
\end{proposition}
\begin{proof}

\begin{gather*}
  \Pr[\text{algorithm does not halt for $v_t$}] = \\
  \Pr[\text{algorithm does not assign a main color to $v_t$
            in any entry}] \leq \\
  \Pr[\text{algorithm does not assign a main color to $v_t$
            in any empty entry}] = \\
  \Pr[\bigcap_{i\colon \text{empty entry}}
        (\text{algorithm does not assign a main color to
               $v_t$ in entry $i$})] =
                                                \displaybreak[0]\\
  \prod_{i\colon \text{empty entry}}
      \Pr[\text{algorithm does not assign a main color to
                $v_t$ in entry $i$}] =
                                                \displaybreak[0]\\
  \prod_{i\colon \text{empty entry}} (1-h^{-1}) =
  \lim_{j \to \infty} (1-h^{-1})^j = 0
\end{gather*}
and therefore \(\Pr[\text{algorithm halts for $v_t$}] = 1\).\qquad
\end{proof}

We proceed to the analysis of the number of colors used by the
algorithm, proving theorem~\ref{degenerate_main}.

\begin{lemma}
\label{expectations} Let $H=(V,E)$ be a hypergraph and let $C$ be
a coloring produced by the above algorithm on an online input
$V=\{v_t\}$ for $t=1,\ldots,n$. Let $X_i$ (respectively $X_{\geq
i}$) denote the random variable counting the number of points of
$V$ that were assigned a final color at entry $i$ (respectively a
final color at some entry $\geq i$). Let $\Ex_i = \Ex[X_i]$ and
$\Ex_{\geq i} = \Ex[X_{\geq i}]$ (note that $X_{\geq i+1} =
X_{\geq i} -X_i$). Then:
$$
\Ex_{\geq i} \leq   \left(\frac{4k+1}{4k+2} \right)^{i-1}n .
$$
\end{lemma}

\begin{proof}
By induction on $i$. The case $i=1$ is trivial. Assume that the
statement holds for $i$. To complete the induction step, we need
to prove that $\Ex_{\geq i+1} \leq (\frac{4k+1}{4k+2})^i n$. By
the conditional expectation formula, we have for any two random
variables $X$, $Y$ that $\Ex[X] = \Ex[\Ex[X \mid Y]]$. Thus,
$$
\Ex_{\geq i+1} = \Ex[\Ex[X_{\geq i+1} \mid X_{\geq i}]] =
\Ex[\Ex[X_{\geq i} - X_i \mid X_{\geq i}]] = \Ex[X_{\geq i} -
\Ex[X_i \mid X_{\geq i}]].
$$

It is easily seen that $\Ex[X_i \mid X_{\geq i}] \geq
\frac{1}{2}\frac{X_{\geq i}}{2k+1}$ since at least half of the
vertices of $X_{\geq i}$ got an auxiliary color by the above
algorithm. Moreover each of those elements that got an auxiliary
color had probability $\frac{1}{2k+1}$ to get the final color $i$.
This is the only place where we need to assume that the adversary
is oblivious and does not have access to the random bits. Thus,
\begin{multline*}
\Ex[X_{\geq i} - \Ex[X_i \mid  X_{\geq i}]] \leq \Ex[X_{\geq i} -
\frac{1}{2(2k+1)}X_{\geq i}] = \frac{4k+1}{4k+2}\Ex[X_{\geq i}] \leq
\left(\frac{4k+1}{4k+2}\right)^i n ,
\end{multline*}
by linearity of expectation and by the induction hypotheses. This
completes the proof of the lemma.\qquad
\end{proof}

\begin{lemma}
\label{expected_colors} The expected number of colors used by the
above algorithm is at most $\log_{\frac{4k+2}{4k+1}}n +1$.
\end{lemma}
\begin{proof}
Let $I_i$ be the indicator random variable for the following
event: some points are colored with a main color in entry $i$. We
are interested in the number of colors used, that is $Y :=
\sum_{i=1}^{\infty} I_i$. Let $b(k,n)=\log_{\frac{4k+2}{4k+1}} n$.
Then,
\[
\Ex[Y]
   =     \Ex[\sum_{1 \le i} I_i]
   \leq  \Ex[\sum_{1 \le i \le b(k,n)} I_i]
       + \Ex[X_{\geq b(k,n)+1}]
 \le b(k,n) + 1  ,
\]
by Markov's inequality  and lemma~\ref{expectations}.\qquad
\end{proof}

We notice that:
\[
b(k,n) =
  \frac{\ln n}{\ln\frac{4k+2}{4k+1}} \leq
  (4k+2)\ln n =
  O( k \log n) .
\]

We also have the following concentration result:
\begin{multline*}
\Pr[\mbox{more than $c \cdot b(k,n)$ colors are used}] = \Pr[X_{\geq c \cdot b(k,n) + 1} \geq 1] \leq \Ex_{\geq c \cdot
b(k,n) + 1}
 \leq \frac{1}{n^{c-1}} ,
\end{multline*}
by Markov's inequality and by lemma~\ref{expectations}.

This completes the performance analysis of the algorithm.

\paragraph{Remark}
In the above description of the algorithm, all the random bits
are chosen in advance (by deciding the values of the function $f$
in advance). However, one can be more efficient and calculate the
entry $f(i)$ only at the first time we need to update entry $i$,
for any $i$. Since at each entry we need to use $O(\log k)$
random bits
and we showed that the number of entries used is $O(k \log n)$
with high probability then the total number of random bits used
by the algorithm is $O(k \log k \log n)$ with high probability.

\subsubsection*{Acknowledgments.}
I would like to thank the two anonymous referees for providing very valuable comments and
suggestions.

\bibliographystyle{abbrv}
\bibliography{survey}

\end{document}

%% file: example.pstex_t
\begin{picture}(0,0)%
\includegraphics{example.pstex}%
\end{picture}%
\setlength{\unitlength}{4144sp}%
\begingroup\makeatletter\ifx\SetFigFont\undefined
\def\x#1#2#3#4#5#6#7\relax{\def\x{#1#2#3#4#5#6}}%
\expandafter\x\fmtname xxxxxx\relax \def\y{splain}%
\ifx\x\y   
\gdef\SetFigFont#1#2#3{%
  \ifnum #1<17\tiny\else \ifnum #1<20\small\else
  \ifnum #1<24\normalsize\else \ifnum #1<29\large\else
  \ifnum #1<34\Large\else \ifnum #1<41\LARGE\else
     \huge\fi\fi\fi\fi\fi\fi
  \csname #3\endcsname}%
\else
\gdef\SetFigFont#1#2#3{\begingroup
  \count@#1\relax \ifnum 25<\count@\count@25\fi
  \def\x{\endgroup\@setsize\SetFigFont{#2pt}}%
  \expandafter\x
    \csname \romannumeral\the\count@ pt\expandafter\endcsname
    \csname @\romannumeral\the\count@ pt\endcsname
  \csname #3\endcsname}%
\fi
\fi\endgroup
\begin{picture}(2422,1479)(523,-1702)
\end{picture}

%% file: 2stalg.pstex_t
\begin{picture}(0,0)%
\includegraphics{2stalg.pstex}%
\end{picture}%
\setlength{\unitlength}{3158sp}%
\begingroup\makeatletter\ifx\SetFigFont\undefined%
\gdef\SetFigFont#1#2#3#4#5{%
  \reset@font\fontsize{#1}{#2pt}%
  \fontfamily{#3}\fontseries{#4}\fontshape{#5}%
  \selectfont}%
\fi\endgroup%
\begin{picture}(4587,1243)(574,-1085)
\put(5161,-1013){\makebox(0,0)[lb]{\smash{{\color[rgb]{0,0,0}level}%
}}}
\put(1958, 14){\makebox(0,0)[lb]{\smash{{\color[rgb]{0,0,0}new point gets level 4}%
}}}
\put(4809,-669){\makebox(0,0)[lb]{\smash{{\color[rgb]{0,0,0}$3$}%
}}}
\put(4808,-1036){\makebox(0,0)[lb]{\smash{{\color[rgb]{0,0,0}$1$}%
}}}
\put(4809,-856){\makebox(0,0)[lb]{\smash{{\color[rgb]{0,0,0}$2$}%
}}}
\end{picture}